\documentclass{amsproc}

\usepackage{verbatim} 
\usepackage{amsmath}
\usepackage{amsfonts}
\usepackage{amssymb}
\usepackage{mathrsfs}
\usepackage{amsthm}
\usepackage{newlfont}
\usepackage{datetime}

\usepackage{fullpage,amsmath, amssymb,amscd}
\usepackage{latexsym, epsfig, color}
\usepackage[mathscr]{eucal}
\usepackage{xypic, graphicx}
\usepackage{amscd}
\usepackage[all, cmtip]{xy}
\usepackage{tikz,tikz-cd,xcolor}
\usepackage{tikz-cd}

\usepackage{scalerel,stackengine}
\stackMath
\newcommand\reallywidehat[1]{%
\savestack{\tmpbox}{\stretchto{%
  \scaleto{%
    \scalerel*[\widthof{\ensuremath{#1}}]{\kern-.6pt\bigwedge\kern-.6pt}%
    {\rule[-\textheight/2]{1ex}{\textheight}}%WIDTH-LIMITED BIG WEDGE
  }{\textheight}% 
}{0.5ex}}%
\stackon[1pt]{#1}{\tmpbox}%
}
\usepackage{float, hyperref, nccmath, multirow, caption, subcaption}

\newcommand\cyr{%
 \renewcommand\rmdefault{wncyr}%
 \renewcommand\sfdefault{wncyss}%
 \renewcommand\encodingdefault{OT2}%
\normalfont\selectfont} \DeclareTextFontCommand{\textcyr}{\cyr}

\newtheorem{theorem}{Theorem}

\newtheorem{proposition}[theorem]{Proposition}
\newtheorem{remark}[theorem]{Remark}

\def\Z{\mathbb Z}

\def\Q{\mathbb Q}
\def\R{\mathbb R}
\def\C{\mathbb C}
\def\F{\mathbb F}

\def\Z{\mathbb Z}

\def\Q{\mathbb Q}
\def\R{\mathbb R}
\def\C{\mathbb C}
\def\F{\mathbb F}

\def\Ker{\operatorname{Ker}}
\def\Im{\operatorname{Im}}
\def\Re{\operatorname{Re}}
\def\deg{\operatorname{deg}}
\def\det{\operatorname{det}}
\def\diag{\operatorname{diag}}

\def\tr{\operatorname{tr}}

\def\End{\operatorname{End}}

\def\Aut{\operatorname{Aut}}
\def\Gal{\operatorname{Gal}}

\def\Frob{\operatorname{Frob}}

\def\mod{\operatorname{mod}}

\def\disc{\operatorname{disc}}
\def\car{\operatorname{char}}
\def\exp{\operatorname{exp}}

\def\gcd{\operatorname{gcd}}

\def\GL{\operatorname{GL}}
\def\GSp{\operatorname{GSp}}

\def\USp{\operatorname{USp}}
\def\PGSp{\operatorname{PGSp}}

\def\O{\operatorname{O}}
\def\o{\operatorname{o}}

\def\log{\operatorname{log}}

\def\ds{\displaystyle}

\def\Ker{\operatorname{Ker}}
\def\Im{\operatorname{Im}}
\def\Re{\operatorname{Re}}
\def\deg{\operatorname{deg}}
\def\det{\operatorname{det}}
\def\diag{\operatorname{diag}}

\def\tr{\operatorname{tr}}

\def\End{\operatorname{End}}

\def\Aut{\operatorname{Aut}}
\def\Gal{\operatorname{Gal}}

\def\Frob{\operatorname{Frob}}

\def\mod{\operatorname{mod}}

\def\disc{\operatorname{disc}}
\def\car{\operatorname{char}}
\def\exp{\operatorname{exp}}

\def\gcd{\operatorname{gcd}}

\def\GL{\operatorname{GL}}
\def\GSp{\operatorname{GSp}}

\def\USp{\operatorname{USp}}

\def\O{\operatorname{O}}
\def\o{\operatorname{o}}

\def\log{\operatorname{log}}

\def\ds{\displaystyle}

\begin{document}

\title{
Bounds for the distribution of the Frobenius traces associated to a generic abelian variety 
}

%\today

%\currenttime 

\author{
Alina Carmen Cojocaru and Tian Wang}
\address[Alina Carmen  Cojocaru]{
\begin{itemize}
\item[-]
Department of Mathematics, Statistics and Computer Science, University of Illinois at Chicago, 851 S Morgan St, 322
SEO, Chicago, 60607, IL, USA;
\item[-]
Institute of Mathematics  ``Simion Stoilow'' of the Romanian Academy, 21 Calea Grivitei St, Bucharest, 010702,
Sector 1, Romania
\end{itemize}
} \email[Alina Carmen  Cojocaru]{cojocaru@uic.edu}

\address[Tian Wang]{
\begin{itemize}
\item[-]
Department of Mathematics \&
Statistic, Concordia University, Montreal, CA;
\end{itemize}
} \email[Tian Wang]{tian.wang@concordia.ca}

\renewcommand{\thefootnote}{\fnsymbol{footnote}}
\renewcommand{\thefootnote}{\arabic{footnote}}

\renewcommand{\thefootnote}{\fnsymbol{footnote}}
\renewcommand{\thefootnote}{\arabic{footnote}}

\thanks{A.C.C. was partially supported  by a Collaboration Grant for Mathematicians from the Simons Foundation  
under Award No. 709008. }

\begin{abstract}
Let $A$ be an abelian variety defined over $\Q$ and of dimension $g$. 
Assume  that, for each sufficiently large prime $\ell$,
$A$ has a surjective residual modulo $\ell$ Galois representation.
For  $t\in \Z$ and  $x>0$, 
denote by $\pi_A(x, t)$ the number of primes $p \leq x$ 
 for which the Frobenius trace $a_{1, p}(A)$  associated to $A (\mod p)$  equals $t$.
Assuming the Generalized Riemann Hypothesis for Dedekind zeta functions (GRH), 
we obtain that
$\pi_A(x, 0) \ll_A x^{1 - \frac{1}{2g^2+g+1}}/(\log x)^{1 - \frac{2}{2g^2+g+1}}$
and 
$\pi_A(x, t) \ll_A x^{1 - \frac{1}{2g^2+g+2}}/(\log x)^{1 - \frac{2}{2g^2+g+2}}$ if $t \neq 0$,
and
 deduce that
almost all primes $p$  satisfy
$|a_{1, p}(A)| > p^{\frac{1}{2 g^2 + g + 1}}/ (\log p)^{\frac{2}{2g^2+g+1}+\varepsilon}$ for any  $\varepsilon>0$.
Assuming, in addition to GRH, 
Artin's Holomorphy Conjecture and a Pair Correlation Conjecture for Artin L-functions,
we obtain that
$\pi_A(x, 0) \ll_A x^{1 - \frac{1}{g+1}}/(\log x)^{1 - \frac{4}{g+1}}$
and 
$\pi_A(x, t) \ll_A x^{1 - \frac{1}{g+2}}/(\log x)^{1 - \frac{4}{g+2}}$ if $t \neq 0$,
and deduce that
 almost all primes $p$  satisfy
$|a_{1, p}(A)|>  p^{\frac{1}{g + 2} - \varepsilon }$ for any  $\varepsilon>0$.
\end{abstract}

\maketitle

\section{Introduction}
%#############################
%#############################
%#############################

In 1976, Lang and Trotter  \cite{LaTr76} proposed a prime distribution problem, which remains open,
 in the setting of
elliptic curves.
Known as the Lang-Trotter Conjecture on Frobenius traces,
the problem 
 aims to understand the reduction type of a given elliptic curve defined over the field of rational numbers.
It may
be described in terms of the distribution of Frobenius elements in 
an infinite family of Galois number fields derived from the given elliptic curve. This description 
may
then  be generalized
 to prime distribution problems related to compatible systems of Galois representations, 
 such as those associated to modular forms and those associated  to higher dimensional abelian varieties. For example, such generalizations were proposed in 
\cite{AkPa19},
\cite{BaGo97}, 
\cite{ChJoSe20},
\cite{CoDaSiSt17}, 
\cite{Ka09},
and
\cite{Mu99}. 
The goal of our paper is to prove results related to the generalization of the Lang-Trotter Conjecture on Frobenius traces
formulated in \cite{CoDaSiSt17} in the setting of generic abelian varieties defined over $\Q$, as explained below.

Let $A$ be an abelian variety defined over $\Q$, of dimension $g \geq 1$, and of conductor $N_A$.
For a rational prime $\ell$, consider
the group
representation obtained from the action of the absolute Galois group $\Gal(\overline{\Q}/\Q)$
on the $\ell$-adic Tate module $T_{\ell}(A)$ of $A$.
Recalling that there is an isomorphism of $\Z_{\ell}$-modules 
$
T_{\ell}(A) \simeq_{\Z_{\ell}} (\Z_{\ell})^{2g}
$
and  fixing a $\Z_{\ell}$-basis of $T_{\ell}(A)$, 
we obtain a Galois representation
\begin{equation*}\label{ell-adic-repres}
\rho_{A, \ell}: \Gal(\overline{\Q}/\Q) \longrightarrow \GL_{2g}(\Z_{\ell}),
\end{equation*}
which we call the $\ell$-adic Galois representation  of $A$. 
Taking the projection of  $\GL_{2g}(\Z_{\ell})$ onto $\GL_{2g}(\Z/\ell \Z)$, we obtain 
a Galois representation
\begin{equation*}\label{ell-adic-repres}
\overline{\rho}_{A, \ell}: \Gal(\overline{\Q}/\Q) \longrightarrow \GL_{2g}(\Z/\ell \Z),
\end{equation*}
which we call
the residual modulo $\ell$ Galois representation of $A$.

For a rational prime $p \nmid \ell N_A$, we denote by
$
a_{1, p}(A)
$
the trace of $\rho_{A, \ell}(\Frob_p)$, where $\Frob_p \in \Gal(\overline{\Q}/\Q)$ is a fixed Frobenius element at $p$. 
It is known that $a_{1, p}(A)$ is an integer which does not depend on $\ell$ and which satisfies the
Hasse-Weil bound
$|a_{1, p}(A)| < 2 g p^{\frac{1}{2}}$.

In 
\cite{LaTr76},
Lang and Trotter  considered the case 
of an abelian variety over $\Q$ of dimension 1 and having a trivial $\overline{\Q}$-endomorphism ring,
and
proposed the investigation  of the asymptotic behaviour of the function counting primes $p$ for which $a_{1, p}(A)$ equals a given integer.
This problem may  also be considered when $g >1$, as follows.

Let $t \in \Z$ and $x \in \R$ with $x > 0$, and define
\begin{equation}\label{pi-Frob-trace}
\pi_{A}(x, t)
:=
\#\left\{
p \leq x: p \nmid N_A, a_{1, p}(A) = t
\right\}.
\end{equation}
Assume, for simplicity, that $A$ is principally polarized, 
in which case
 the adelic Galois representation  $\rho_A$,  
 defined as the product of the $\ell$-adic Galois representations $\rho_{A, \ell}$,
 has its image contained in $\GSp_{2g}(\hat{\Z})$.
If 
 $\rho_{A}$
has open image in $\GSp_{2g}(\hat{\Z})$ 
and
if the normalized traces $\frac{a_{1, p}(A)}{\sqrt{p}}$ are equidistributed on the interval $[-2g, 2g]$
with respect to the projection by the trace map of the normalized Haar measure on the unitary symplectic
group $\USp_{2g}$, then
Cojocaru, Davis, Silverberg and Stange  \cite{CoDaSiSt17}  conjectured that,
provided $t \neq 0$ when $g \geq 2$,
there exists an explicit constant $c(A, t) \geq 0$ such that,
 as $x \rightarrow \infty$,
\begin{equation}\label{lang-trotter-av}
\pi_A(x, t) \sim c(A, t) \frac{\sqrt{x}}{\log x}.
\end{equation}
In the case of an abelian variety of dimension
$g = 1$ 
and having a trivial $\overline{\Q}$-endomorphism ring,
both the open image assumption and the equidistribution assumption are known to hold
and 
the asymptotic formula
 (\ref{lang-trotter-av}) coincides with the Lang-Trotter Conjecture on Frobenius traces proposed in 
\cite{LaTr76}.
In the case of an abelian variety of dimension
$g \geq 2$, 
the open image assumption is known to hold 
if $g =2, 6,$ or odd, and the $\overline{\Q}$-endomorphism ring of $A$ is trivial
(see \cite{Se85}, \cite{Se86});
in contrast,
the equidistribution assumption
-- a special case of a broad conjecture that generalizes the Sato-Tate Conjecture  (see \cite{Se94}  and \cite{KaSa99}) --
 remains open.

The growth of the function $\pi_A(x, t)$ is intimately connected to
the distribution of Frobenius elements in extensions of $\Q$ derived from $A$.
For example, the Chebotarev Density Theorem,
applied in extensions of $\Q$ associated to the division fields of $A$, 
plays a fundamental role in the study 
of $\pi_A(x, t)$. 
Since the best effective versions of this theorem
depend on the assumption of
a Generalized Riemann Hypothesis for Dedekind zeta functions (denoted GRH, for short),
it is natural to investigate the growth of $\pi_A(x, t)$ under such an assumption.

Currently,
the best results about the growth  of $\pi_A(x, t)$, obtained under some GRH,
are as follows:
provided that $g = 1$ and  $A$ has a trivial $\overline{\Q}$-endomorphism ring,
it was proven in \cite{MuMuSa88} and \cite{Zy15} that
\begin{equation}\label{ec-GRH}
\pi_A(x, t) 
\ll_{A}
\left\{
\begin{array}{cl}
\frac{x^{1 - \frac{1}{5}}}{(\log x)^{1 -\frac{2}{5}}}  & \text{if  $t \neq 0$,}
\\
\\
\frac{x^{1 - \frac{1}{4}}}{ (\log x)^{1 - \frac{1}{2}} } & \text{if  $t =0$};
\end{array}
\right.
\end{equation}
provided that
$g \geq 2$ and
 $\rho_{A}$
has open image in $\GSp_{2g}(\hat{\Z})$, 
it was proven in \cite{CoDaSiSt17} that,
for any $\varepsilon > 0$,
 \begin{equation}\label{av-not-zero}
\pi_A(x, t) 
\ll_{A, \varepsilon}
\left\{
\begin{array}{cl}
x^{1 - \frac{1}{2 (2g^2 - g + 3)} + \varepsilon}   & \text{if $t \neq \pm 2g, 0$,}
\\
\\
x^{1 - \frac{1}{2 (2g^2 + g + 1)} + \varepsilon} & \text{if $t = \pm 2g$,}
\end{array}
\right.
\end{equation}
and
\begin{equation}\label{av-zero}
\pi_A(x, 0) 
\ll_{A, \varepsilon}
\left\{
\begin{array}{cl}
x^{1 - \frac{1}{16} + \varepsilon}   & \text{if $g = 2$,}
\\
\\
x^{1 - \frac{1}{2 (2g^2 - g + 1)} + \varepsilon} & \text{if $g \geq 3$.}
\end{array}
\right.
\end{equation}

 Our main goal in this paper is to improve upon the above results when $g \geq 2$.
 Specifically, under a GRH assumption, we prove the following upper bounds for $\pi_A(x, t)$.

\begin{theorem}\label{main-thm1}
Let $t \in \Z$
and
let $A$ be an abelian variety  defined over $\Q$ and
of dimension $g$.
Assume that, for any sufficiently large prime $\ell$,
the  image of the residual modulo $\ell$ Galois representation $\overline{\rho}_{A, \ell}$ of $A$
is isomorphic to $\GSp_{2g}(\Z/\ell \Z)$.
Assume GRH.
Then, for any sufficiently large $x$, we have
$$
\pi_A(x, t) 
\ll_{A}
\left\{
\begin{array}{cl}
\frac{x^{1 - \frac{1}{2g^2+g+2}}}{(\log x)^{1 - \frac{2}{2g^2+g+2}}}  & \text{if $t \neq 0$,}
\\
\\
\frac{x^{1 - \frac{1}{2g^2+g+1}}}{ (\log x)^{1 - \frac{2}{2g^2+g+1}} } & \text{if $t =0$.}
\end{array}
\right.
$$
\end{theorem}

Note that the bounds of Theorem \ref{main-thm1}  
recover (\ref{ec-GRH}) when $g = 1$ and  largely improve upon (\ref{av-not-zero}) - (\ref{av-zero}) when  $g \geq 2$.
When $g = 2$, Theorem \ref{main-thm1}  also
matches a recent upper bound presented  in \cite{KuKuWe22} which, by \cite{BoCaGePi21}, 
 is applicable to a modular abelian surface defined over $\Q$ and having a trivial endomorphism ring.

We remark that the assumption made in Theorem \ref{main-thm1}
about the  image of the residual modulo $\ell$ Galois representation $\overline{\rho}_{A, \ell}$ of $A$
being isomorphic to $\GSp_{2g}(\Z/\ell \Z)$ is weaker than,
but implied by, the open image assumption of (\ref{lang-trotter-av}).
As mentioned earlier,
this latter assumption is known to hold for all abelian varieties over $\Q$ of dimension  $g = 1, 2, 6,$ or odd and having  a trivial $\overline{\Q}$-endomorphism ring.
Moreover, it is known to hold
for  any abelian $g$-fold 
that arises as  the Jacobian of  a hyperelliptic curve defined by 
$Y^2 = f(X)$ for some monic polynomial $f(X) \in \Z[X]$ of degree $\deg f \in \{2 g + 1, 2 g + 2\}$
for some $g \geq 2$
and having
  the property that
either
the Galois group of $f$ is the permutation group on $\deg f$ elements,
or 
there exists a rational prime $p$ for which $f(X) (\mod p)$ has $\deg f - 1$ distinct roots over an algebraic closure of $\F_p$,
one of which is a double root (see the results of Hall and Kowalski in \cite{Ha11} and of Zarhin in \cite{Za00}).
Consequently, the strongest assumption of Theorem \ref{main-thm1} is that of a GRH.

 Our secondary goal in this paper is to improve the bounds of Theorem \ref{main-thm1} under additional hypotheses.
 It is known that
 if, in addition to a GRH, one assumes further hypotheses on the Artin L-functions associated to a
 given Galois extension,
then one may improve  the error term in the
 effective version of the Chebotarev Density Theorem pertaining to that extension. 
 The impact of such improvements on the 
 growth of $\pi_A(x, t)$ motivated the bounds (\ref{ec-GRH}) of \cite{MuMuSa88} and \cite{Zy15} when $g = 1$, which, in turn, motivated the following results of Bella\"{i}che \cite{Be16} when $g \geq 2$:
assume that, for any sufficiently large prime $\ell$,
the  image of the $\ell$-adic Galois representation $\rho_{A, \ell}$ of $A$
is isomorphic to $\GSp_{2g}(\Z_\ell)$;
assume GRH;
assume  Artin's Holomorphy Conjecture for Artin L-functions (denoted AHC, for short);
then,
 for any $\varepsilon > 0$,
 \begin{equation}\label{av-Bellaiche}
\pi_A(x, t) 
\ll_{A, \varepsilon}
\left\{
\begin{array}{cl}
x^{1 - \frac{2}{2g^2 - g + 4} + \varepsilon}   & \text{if $t \neq 0$,}
\\
\\
x^{1 - \frac{2}{2g^2 + g + 3} + \varepsilon} & \text{if $t = 0$.}
\end{array}
\right.
\end{equation}
In \cite{MuMuWo18},
Murty, Murty, and Wong
improved the error term in the
 effective version of the Chebotarev Density Theorem even further 
by assuming, in addition to GRH and  AHC,
a Pair Correlation Conjecture for the Artin L-functions (denoted PCC, for short) associated to the given Galois extension.
As an application, they obtained the following 
 improvement to (\ref{ec-GRH}):
  assume that $g = 1$ and that $A$ has a trivial $\overline{\Q}$-endomorphism ring;
assume GRH;
assume   AHC and  PCC;
then
 \begin{equation}\label{ec-MuMuWo}
\pi_A(x, t) 
\ll_{A}
\left\{
\begin{array}{cl}
x^{1 - \frac{1}{3}} (\log x)^{\frac{1}{3}}   & \text{if $t \neq 0$,}
\\
\\
x^{1 - \frac{1}{2}} \log x & \text{if $t = 0$.}
\end{array}
\right.
\end{equation}

In this paper, we generalize (\ref{ec-MuMuWo}) to the case of an abelian variety of higher dimensions
and
obtain the following upper bounds for $\pi_A(x, t)$.

\begin{theorem}\label{main-thm2}
Let $t \in \Z$
and
let $A$ be an abelian variety defined over $\Q$ and
of dimension $g$.
Assume that, for any sufficiently large prime $\ell$,
the  image of the residual modulo $\ell$ Galois representation $\overline{\rho}_{A, \ell}$ of $A$
is isomorphic to $\GSp_{2g}(\Z/\ell \Z)$.
Assume  GRH, AHC, and PCC. 
Then, for any sufficiently large $x$, we have
$$
\pi_A(x, t) 
\ll_{A}
\left\{
\begin{array}{cl}
\frac{x^{1 - \frac{1}{g+2}}}{(\log x)^{1 - \frac{4}{g+2}}}  & \text{if $t \neq 0$,}
\\
\\
\frac{x^{1 - \frac{1}{g+1}}}{(\log x)^{1 - \frac{4}{g+1}}} & \text{if $t =0$.}
\end{array}
\right.
$$
\end{theorem}

By using  Theorems \ref{main-thm1} and \ref{main-thm2}, 
we obtain the following non-trivial 
lower bounds for 
$|a_{1, p}(A)|$ 
for almost all $p$.

\begin{theorem}\label{main-thm3}
Let $A$ be an abelian variety defined over $\Q$ and
of dimension $g$.
Assume that, for any sufficiently large prime $\ell$,
the  image of the residual modulo $\ell$ Galois representation $\overline{\rho}_{A, \ell}$ of $A$
is isomorphic to $\GSp_{2g}(\Z/\ell \Z)$.
\begin{enumerate}
\item[(i)]
Assume GRH.
Then,  for any $\varepsilon >0$, the inequality
\[
|a_{1, p}(A)| > \frac{ p^{\frac{1}{2 g^2 + g + 1}}  }{ (\log p)^{\frac{2}{2g^2+g+1}+\varepsilon}   }
\]
holds for a set of primes $p$ of density 1.
\item[(ii)]
Assume  GRH, AHC, and PCC. 
Then, for any $\varepsilon >0$, the inequality
\[
|a_{1, p}(A)| >p^{\frac{1}{g + 2}-\varepsilon}
\]
holds for a set of primes $p$ of density 1.
\end{enumerate}
\end{theorem}

The motivation for Theorem \ref{main-thm3} comes from a conjecture of Atkin and Serre  on a lower bound
for $\tau(p)$
for the Ramanujan $\tau$-function (see \cite{Se76}). 
To see this,
recall that,
when
$g = 1$, 
the integer $a_{1, p}(A)$ 
coincides with the $p$-th Fourier coefficient of the weight two normalized newform, of level $N_A$, associated to $A$ under the Shimura-Taniyama-Weil Conjecture. 
Using this point of view, 
Theorem \ref{main-thm3} relates to results on the growth of the $p$-th Fourier coefficients of newforms
pursued in \cite{MuMu84} and \cite{MuMuSa88},
which had been
motivated by the aforementioned conjecture of  Atkin and Serre.

Before concluding  
the introduction,
let us comment briefly on the proofs of our first two main results.
The proofs of Theorems \ref{main-thm1} - \ref{main-thm2} have their roots in  
the observation that 
if $a_{1, p}(A) = t$, then 
$a_{1, p}(A) \equiv t (\mod \ell)$ 
for any prime $\ell$,
and in 
the interpretation of $a_{1, p}(A) (\mod \ell)$ as the trace of the image of $\Frob_p$ under
the residual modulo $\ell$ Galois representation of $A$.
An application of some version of the Chebotarev  Density Theorem should then lead to some non-trivial upper bound
for $\pi_A(x, t)$. For example, when $g = 1$ and $A$ has a trivial $\overline{\Q}$-endomorphism ring,
by assuming GRH and  using the effective version of the Chebotarev Density Theorem of Lagarias and Odlyzko \cite{LaOd77},
one obtains the upper bound
$\pi_A(x, t) \ll_{A, \varepsilon} x^{1 - \frac{1}{8} + \varepsilon}$.
While this mod $\ell$ approach easily leads to a non-trivial upper bound for $\pi_A(x, t)$ when $g = 1$, 
refining it to improve upon this bound 
or generalizing it to obtain a non-trivial upper bound for $\pi_A(x, t)$ when $g \geq 2$ 
requires further tools and new ideas.

One strategy of refining the mod $\ell$ approach is to consider, simultaneously, $\mod \ell^k$ congruences, 
for some fixed prime $\ell$ and varying positive integers $k \geq 1$.
For example, in \cite{Se81},
Serre considered  the case $g = 1$ and $A$ has a trivial $\overline{\Q}$-endomorphism ring,
and related
the problem of estimating $\pi_A(x, t)$
 to a prime counting problem
in the setting of an infinite Galois extension of $\Q$
having the $\ell$-adic Lie group $\GL_2(\Z_{\ell})$ as its Galois group, 
where $\ell$
 is some sufficiently large prime defined in terms of $x$.
Then, he used  an $\ell$-adic version of the Chebotarev Density Theorem
 to estimate, from above, the emerging set of primes.
 In particular, by assuming GRH, Serre obtained the upper bound
 $\pi_A(x, t) \ll_{A, \varepsilon} x^{1 - \frac{1}{6} + \varepsilon}$.
 This $\ell$-adic approach was generalized to $g \geq 2$ 
 by Cojocaru, Davis, Silverberg, and Stange in \cite{CoDaSiSt17}, 
 who proved the upper bounds
 (\ref{av-not-zero}) and (\ref{av-zero}), under GRH.
 
Another strategy of refining the mod $\ell$ approach is 
to improve the error term in the effective version of the Chebotarev Density Theorem.
For example, 
 in  \cite{MuMuSa88},
by assuming, 
both GRH and AHC,
Murty, Murty, and Saradha 
proved a refined version of the effective Chebotarev Density Theorem of \cite{LaOd77}.
Considering the case $g = 1$ and $A$ has a trivial $\overline{\Q}$-endomorphism ring,
a direct application of this theorem
in the mod $\ell$ approach
 would lead to the upper bound
 $\pi_A(x, t) \ll_{A, \varepsilon} x^{1 - \frac{1}{5} + \varepsilon}$,
 under GRH and AHC.
To circumvent the assumption
of
 AHC
 and still obtain this upper bound, 
Murty, Murty, and Saradha
related
 the problem of estimating $\pi_A(x, t)$
  to a prime counting problem
 in the setting of a particular subextension $L/K$ of number fields,
 having its Galois group isomorphic to the maximal torus of $\GL_2(\F_{\ell})$ (see \cite[pp. 271--272]{MuMuSa88}), 
where $\ell$ is some sufficiently large prime defined in terms of $x$.
Then they 
estimated the emerging set of primes
by appealing to their improved effective version of the Chebotarev Density Theorem,
assuming only GRH,
since 
AHC is known in an abelian setting.
Overall, their refined mod $\ell$ approach, together with an improvement of \cite{Zy15},
 led to the upper bounds (\ref{ec-GRH}).
 In \cite{Be16}, Bella\"{i}che refined the aforementioned
improved version of the effective Chebotarev Density Theorem of Murty, Murty, and Saradha in certain cases 
and succeeded in generalizing the proof of (\ref{ec-GRH})
in the case $g = 2$  and $A$ has a trivial $\overline{\Q}$-endomorphism ring,
proving, under both GRH and AHC, that
$\pi_A(x, 0) \ll_A \frac{x^{1 - \frac{1}{10}}}{(\log x)^{1 - \frac{2}{5}}}$.
Unlike the final bound of Murty, Murty, and Saradha, which only 
assumes
GRH,
 that of Bella\"{i}che 
 assumes
   both GRH and AHC.
 
One could also relate mod $\ell$ approach of estimating $\pi_A(x, t)$ to a sieve, such as the large sieve,
and then use an improved version of the effective Chebotarev Density Theorem to execute the sieve
application. This  approach was elaborated on in \cite{Be16}, leading to the improved upper bounds
(\ref{av-Bellaiche}), which hold for $g \geq 1$ and which 
 assume
 both GRH and AHC.

In the case $g = 1$ and $A$ has a trivial $\overline{\Q}$-endomorphism ring,
the mod $\ell$ approach witnessed yet another refinement in \cite{MuMuWo18}.
Therein,
by assuming GRH,  AHC, and PCC,
Murty, Murty, and Wong
improved  the effective Chebotarev Density Theorem of \cite{MuMuSa88}.
Then, they used this theorem directly in the mod $\ell$ approach to prove the upper bounds (\ref{ec-MuMuWo}),
under GRH, AHC, and PCC.
Via this strategy, in order to take advantage of the power of the PCC assumption, one cannot
circumvent the assumption of AHC by working in an abelian extension of number fields.

 Our proof of Theorem \ref{main-thm1} generalizes the mod $\ell$ approach  of \cite{MuMuSa88},
 under GRH, 
 to the case $g \geq 2$, overcoming prior obstacles faced in \cite{Be16}
 and \cite{CoDaSiSt17}.
 The main challenge consists 
 of
  unraveling,
  for some sufficiently large prime $\ell$,
  a suitable distinguished pair of subfields of a finite Galois extension of $\Q$ having
 its Galois group isomorphic to 
 an abelian and sufficiently large  subquotient of $\GSp_{2g}(\F_{\ell})$ (see (\ref{thm1-B/U}) and (\ref{thm1-B/U'}),
  and
 of
 finding
 a suitable conjugacy class in this  group (see  $ \widehat{ \cal{C}_{B}(\ell, t)}$ and $\widehat{\cal{C}'_{B}(\ell, t)}$  introduced in Section 5),
for which
  the effective version of the Chebotarev Density Theorem of \cite{MuMuSa88}
 could be applied successfully.
 
 Our proof of Theorem \ref{main-thm2} generalizes the mod $\ell$ approach of \cite{MuMuWo18},
 under GRH, AHC, and PCC,
  to the case $g \geq 2$, overcoming obstacles faced in \cite{CoDaSiSt17}.
 The main challenge was
 that of
  unraveling,
  for some sufficiently large prime $\ell$,
  a suitable conjugacy class in $\GSp_{2g}(\F_{\ell})$
for which
  the  effective version of the Chebotarev Density Theorem of \cite{MuMuWo18}
 could be applied successfully.

The statements of the Generalized Riemann Hypothesis, Artin's Holomorphy Conjecture, and the Pair Correlation Conjecture,
as well as the main notation used in the paper, are given in Section 2.
The effective  versions of the Chebotarev Density Theorem needed for our proofs are presented in Section 3.
The distinguished subgroups and conjugacy classes of $\GSp_{2g}(\F_{\ell})$ needed in our approaches are discussed in Sections 4 - 5.
The proofs of Theorems \ref{main-thm1} - \ref{main-thm3} are given in Section 6, with
Propositions \ref{key-prop} - \ref{max-lemma-A-generic} playing the role of crucial ingredients.
Following these proofs, brief remarks about possible other approaches towards obtaining conditional bounds for 
$\pi_A(x, t)$ are included in Section 7.

\medskip

\noindent
{\bf{Acknowledgments.}} 
We thank Professor Kiran S. Kedlaya and Professor Peter Sarnak for stimulating discussions related to this work. We are grateful for the many helpful
comments and suggestions from the referees.

%\bigskip
%#############################
%#############################
%#############################
\section{Notation}
%#############################
%#############################
%#############################

\noindent
Throughout the paper, we use the following  notation.

\noindent
$\bullet$
Given a finite set $S$, we denote its cardinality by 
$\# S$.

\noindent
$\bullet$
Given suitably defined real functions $h_1, h_2$,
we say that 
$h_1 = \o(h_2)$ if $\ds\lim_{x \rightarrow \infty} \frac{h_1(x)}{h_2(x)} = 0$;
we say that
$h_1 = \O(h_2)$ or, equivalently, that $h_1 \ll h_2$, 
if 
$h_2$ is positive valued 
and
 there exists a positive constant $c$ such that 
$|h_1(x)| \leq c \ h_2(x)$ for all $x$ in the common domain of $h_1$ and $h_2$;
we say that
$h_1 = \O_D(h_2)$ or, equivalently, that $h_1 \ll_D h_2$,
if 
$h_1 = \O(h_2)$
and
the implied $\O$-constant $c$ depends on priorly given data  $D$;
we say that 
$h_1 \sim h_2$ 
if
$\ds\lim_{x \rightarrow \infty} \frac{h_1(x)}{h_2(x)} = 1$.

\noindent
$\bullet$
We use the letters $p$ and $\ell$ to denote rational primes.
We denote by $\pi(x)$ the number of primes $p \leq x$
and recall that, by the Prime Number Theorem,
$\pi(x) \sim \frac{x}{\log x}$.

\noindent 
$\bullet$
 Given a positive integer $m$, we denote by $\Z/m \Z$ the ring of integers modulo $m$.
 When $m$ is a prime $\ell$, we denote $\Z/\ell \Z$ by $\F_{\ell}$ to emphasize 
 its
 field structure.
For an integer $a$, we denote by $a (\mod \ell)$ its residue class modulo $\ell$.

\noindent 
$\bullet$
Given a prime $\ell$, we denote by $\Z_{\ell}$ the ring of $\ell$-adic integers.
We set $\hat{\Z} := \ds \lim_{\leftarrow \atop m} \Z/m \Z$
and recall that there exists a ring isomorphism 
$\hat{\Z}  \simeq \ds\prod_{\ell} \Z_{\ell}$.

\noindent
$\bullet$
 Given a number field $K$, 
 we denote  
 by  ${\cal{O}}_K$ its ring of integers,
  by ${\sum}_K$ the set of non-zero prime ideals of  ${\cal{O}}_K$,
 by $[K:\Q]$ the degree of $K$ over $\Q$,
 by $d_K \in \Z \backslash \{0\}$ the discriminant of an integral basis of ${\cal{O}}_K$,
 and 
 by  $\disc(K/\Q) = \Z d_K \unlhd \Z$ the discriminant ideal of $K/\Q$.
 For a prime ideal $\mathfrak{p} \in {\sum}_K$, 
 we denote by $N_{K/\Q}(\mathfrak{p})$ its norm in $K/\Q$.
 We say that $K$ satisfies the Generalized Riemann Hypothesis  (GRH)  if
 the Dedekind zeta function $\zeta_K$ of $K$ has the property that,
 for any $\rho \in \C$ with $0 \leq \Re \rho \leq 1$ and $\zeta_K(\rho) = 0$, we have $\Re(\rho) = \frac{1}{2}$. When $K=\Q$, the Dedekind zeta function is the Riemann zeta function, in which case we refer to GRH  as the Riemann Hypothesis (RH).

\noindent
$\bullet$
Given  a finite Galois extension $L/K$ of number fields
and given an irreducible character $\chi$ of the Galois group of $L/K$,
we denote 
by $\mathfrak{f}(\chi) \unlhd {\cal{O}}_K$ the global Artin conductor of $\chi$,
by $A_{\chi} := |d_L|^{\chi(1)} N_{K/\Q}(\mathfrak{f}(\chi)) \in \Z$ the conductor of $\chi$,
and by ${\cal{A}}_{\chi}(T)$ the function of a positive real variable $T > 3$ defined by the relation
$$
\log {\cal{A}}_{\chi}(T) = \log A_{\chi} + \chi(1) [K:\Q] \log T.
$$
 
\noindent
$\bullet$
Given  a finite Galois extension $L/K$ of number fields,
we say that it satisfies Artin's Holomorphy Conjecture (AHC)
if, for any irreducible character $\chi$ of the Galois group of $L/K$,
the Artin L-function $L(s, \chi, L/K)$ extends to a function that is
analytic on the whole $\C$, except at $s = 1$ when $\chi = 1$.
We recall that, if
we assume GRH for $L$ and AHC for $L/K$, 
then, given 
any  irreducible character $\chi$ of the Galois group of $L/K$,
and given any non-trivial zero 
$\rho$ of $L(s, \chi, L/K)$, 
 the real part $\Re \rho$ of $\rho$ satisfies
$\Re \rho = \frac{1}{2}$.
In this case, we write $\rho = \frac{1}{2} + i \gamma$, where 
$\gamma$ denotes the imaginary part 
of $\rho$.

\noindent
$\bullet$
Given  a finite Galois extension $L/K$ of number fields,
let us assume GRH for $L$ and AHC for $L/K$.
For an  irreducible character $\chi$ of the Galois group of $L/K$
and an arbitrary $T > 0$, 
we define the pair correlation function of $L(s, \chi, L/K)$ by
\begin{equation*}\label{pair-cor-fcn}
{\cal{P}}_T (X, \chi)
:=
\ds\sum_{- T \leq \gamma_1 \leq T}
\ds\sum_{- T \leq \gamma_2 \leq T}
w(\gamma_1 - \gamma_2)
e((\gamma_1 - \gamma_2) X),
\end{equation*}
where
$\gamma_1$ and $\gamma_2$ range over all the imaginary parts of the non-trivial zeroes
$\rho = \frac{1}{2} +  \gamma$
 of $L(s, \chi, L/K)$,
counted with multiplicity, 
and where,
 for an arbitrary real number $u$,
$e(u) := \exp(2 \pi i u)$ and $w(u) := \frac{4}{4 + u^2}$.
We say that the extension $L/K$ satisfies the Pair Correlation Conjecture (PCC) 
if, 
 for any irreducible character $\chi$ of the Galois group of $L/K$ and for any  
 $c>0$ and $T > 3$,
 provided
$0 \leq Y \leq c \  \chi(1) [K:\Q] \log T$,
we have
$$
{\cal{P}}_T(Y, \chi) \ll_c \chi(1)^{-1} T \log {\cal{A}}_{\chi}(T). 
$$

\noindent
$\bullet$
Given a field $F$, we denote 
by $\car{F}$ its field characteristic,
by  $\overline{F}$ a fixed algebraic closure of $F$,
by $F^{\text{sep}}$ a fixed separable closure of $F$ in $\overline{F}$,
and
by $\Gal(F^{\text{sep}}/F)$ the absolute Galois group of $F$.
 
\noindent
$\bullet$
Given a non-zero unitary commutative ring $R$, we denote by
$R^{\times}$  its group of multiplicative units.

\noindent
$\bullet$
Given 
a non-zero unitary commutative ring $R$
and
a free module ${\cal{M}}$ over  $R$, of finite rank $n \geq 1$,
endowed with a non-degenerate alternating bilinear form 
$e : {\cal{M}} \times {\cal{M}} \rightarrow R$,
we denote by $\GSp({\cal{M}}, e)$
the group of symplectic similitudes of ${\cal{M}}$ with respect to $e$,
that is,
the group of $R$-automorphisms $\sigma\in \Aut({\cal{M}})$
such that
there exists $\mu \in R^{\times}$
with the property that
$e (\sigma(v), \sigma(w)) = \mu \ e(v, w)$ for all $v, w \in {\cal{M}}$.
Since ${\cal{M}}$ has a unique (up to isomorphism) non-degenerate alternating bilinear form, 
we may write $\GSp_{n}(R)$ for $\GSp({\cal{M}}, e)$.

\noindent 
$\bullet$
Given a non-zero unitary commutative ring $R$ and an integer $n \geq 1$, 
we denote by
$M_{n}(R)$
the ring of $n \times n$ matrices with entries in $R$
and by $I_n$ the  identity  matrix in $M_n(R)$.
For an arbitrary matrix $M \in M_n(R)$,
we denote by $\tr M$ and $\det M$ its trace and determinant,
and by $M^t$ its transpose.
We define the general linear group  $\GL_{n}(R)$
 as the collection of  $M \in M_n(R)$ with $\det M \in R^{\times}$.
  For 
$a_1, \ldots, a_g \in R$,
we define
$$
 \text{diag}(a_1, \ldots, a_n)  :=
    \begin{pmatrix} 
  a_{1}  &  & \\
  &  \ddots & \\
  & & a_{n}  
  \end{pmatrix}
  \in
  M_n(R).
  $$
 When $n = 2g$ for some integer $g \geq 1$,
 we define
\[
J_{2g}  :=   \begin{pmatrix} 0 & I_{g} \\  -I_g & 0 \end{pmatrix} \in M_{2 g}(R).
\] 
The general symplectic group over $R$ 
may then be described as
\[
\GSp_{2g}(R) =\left\{M\in \GL_{2g}(R): M^{t}J_{2g}M=\mu J_{2g} \ \text{for some} \ \mu \in R^{\times} \right\}.
\]
Associated to $\GSp_{2g}(R)$, 
we have the  character 
$
\GSp_{2g}(R)  \longrightarrow R^{\times},
$
$M \mapsto \mu$.
The scalar $\mu$ corresponding to $M$ is 
called
 the multiplicator of $M$.
 We write the group of  scalar matrices in $\GSp_{2g}(R)$ as $\Lambda_{2g}(R)$
and note that it is isomorphic to $R^{\times}$.
We define the projective general symplectic group $\PGSp_{2g}(R)$ as the quotient 
$\GSp_{2g}(R)/\Lambda_{2g}(R)$.

%#############################
%#############################
%#############################
\section{Frobenius distributions in number field extensions}
%#############################
%#############################
%#############################

In this section, we record  results on the theme of Frobenius distributions in a number field Galois extension;
these results will play a crucial role in the proofs of Theorems \ref{main-thm1} - \ref{main-thm3}.

Let $L/K$ be a finite Galois extension of number fields.
Denote
 by $\Gal(L/K)$ the Galois group of $L/K$
 and keep the number field notation introduced in Section 2.
 Denote by $\Gal(L/K)^{\#}$ the set of  conjugacy classes of $\Gal(L/K)$.
 Additionally, 
 denote
 by  $[L:K]$ the degree of $L/K$
 and
 by $\disc(L/K)$ the discriminant ideal of $L/K$.

We define
$$
 P(L/K) := \{p \ \text{rational prime}: \ \exists \ \mathfrak{p} \in {\sum}_{K}
 \ \text{such that} \
 \mathfrak{p} \mid p \ \text{and} \ \mathfrak{p} \mid \disc(L/K)\},
 $$
 $$
 M(L/K) := 2 [L:K] |d_K|^{\frac{1}{[K:\Q]}} \ds\prod_{p \in P(L/K)} p,
 $$
and we recall  from \cite[Prop.~5, p.~129]{Se81}  that
\begin{equation}\label{hensel}
\log \left|N_{K/\Q} (\disc (L/K))\right|
\leq
([L : \Q] - [K : \Q])
\left(
\ds\sum_{p \in P(L/K)} \log p
\right)
+
[L : \Q] \log [L : K].
\end{equation}

For a place  $\wp \in \Sigma_L$,
we  denote 
 by ${\cal{D}}_{\wp}$ its decomposition group in $L/K$, 
 by ${\cal{I}}_{\wp}$ its inertia group in $L/K$,
 and
 by $\left(\frac{L/K}{\wp}\right) \in {\cal{D}}_{\wp}/{\cal{I}}_{\wp}$ its Frobenius element in $L/K$.
 
 For a  place  $\mathfrak{p} \in \Sigma_K$,
 we set 
 $$\left(\frac{L/K}{\mathfrak{p}}\right) 
 := \left\{\left(\frac{L/K}{\wp}\right): \wp \in \Sigma_L, \wp \ \text{lies over} \ \mathfrak{p} \right\}.$$
 
For  a non-empty set ${\cal{C}} \subseteq \Gal(L/K)$, stable under conjugation by elements of $\Gal(L/K)$,
we denote by $\delta_{\cal{C}} : \Gal(L/K) \longrightarrow \{0, 1\}$
its characteristic function
and,
for an arbitrary place 
$\mathfrak{p} \in \Sigma_K$
 and an arbitrary integer $m \geq 1$,  we set
$$
\delta_{\cal{C}}\left(\left(\frac{L/K}{\mathfrak{p}}\right)^m\right)
:=
\frac{1}{\#{\cal{I}}_{\wp}} 
\ds\sum_{
\gamma \in {\cal{D}}_{\wp}
\atop{
\gamma {\cal{I}}_{\wp} = \left(\frac{L/K}{\wp}\right)^m \in {\cal{D}}_{\wp}/{\cal{I}}_{\wp}
}
}
\delta_{\cal{C}}(\gamma),
$$
where $\wp \in \Sigma_L$ is an arbitrary place above $\mathfrak{p}$,
whose choice leaves  the above definition unchanged.
For a real number $x \geq 2$, we set
\begin{eqnarray*}
\pi_{\cal{C}}(x, L/K)
&:=&
\ds\sum_{
\mathfrak{p} \in {\sum}_K
\atop{
\mathfrak{p} \nmid \disc(L/K)
\atop{
N_{K/\Q}(\mathfrak{p})\leq x
}
}
}
\delta_{\cal{C}}\left(
\left(\frac{L/K}{\mathfrak{p}}\right)
\right),
\\
\tilde{\pi}_{\cal{C}}(x, L/K)
&:=&
\ds\sum_{m \geq 1}
\frac{1}{m}
\ds\sum_{
\mathfrak{p} \in {\sum}_K
\atop{
N_{K/\Q}(\mathfrak{p}^m) \leq x
}
}
\delta_{\cal{C}}\left(
\left(\frac{L/K}{\mathfrak{p}}\right)^m
\right),
\end{eqnarray*}
and
we recall from \cite[Proposition 7, p. 138]{Se81} and \cite[Lemma 2.7, p. 8]{Zy15} that
\begin{equation}\label{pi-versus-pi-tilde}
\tilde{\pi}_{\cal{C}}(x, L/K) 
= \pi_{\cal{C}}(x, L/K)
+
\O\left(
[K:\Q]
\left(
\frac{x^{\frac{1}{2}}}{\log x} + \log M(L/K)
\right)
\right).
\end{equation}

The variations of the effective version of the Chebotarev Density Theorem of \cite{LaOd77} that relate to
 the proofs of our main theorems are as follows.

\begin{theorem}\label{chebotarev-GRH-AHC-PCC}
Let $L/K$ be a Galois extension of number fields.
Let $\emptyset \neq \cal{C} \subseteq \Gal(L/K)$ 
be a subset stable under conjugation by elements of $\Gal(L/K)$.
\begin{enumerate}
\item[(i)]
Assume GRH for the Dedekind zeta function of $L$. 
Then
\begin{equation*}\label{cheb-LaOd} 
\pi_{\cal{C}}(x, L/K) 
= 
\frac{\#\cal{C}}{[L:K]} \pi(x)
+
\O\left(
(\#\cal{C}) \   x^{\frac{1}{2}} \ [K:\Q]
\left(
\frac{\log |d_L|}{[L:\Q]}
+
 \log x
\right)
\right).
\end{equation*}
\item[(ii)]
Assume GRH for the Dedekind zeta function of $L$ and AHC for the extension $L/K$.
Then
\begin{equation*}\label{cheb-MuMuSa} 
\pi_{\cal{C}}(x, L/K) 
= 
\frac{\#\cal{C}}{[L:K]} \pi(x)
+
\O\left(
(\#\cal{C})^{\frac{1}{2}} \   x^{\frac{1}{2}} \ [K:\Q]
\log (M(L/K) x)
\right).
\end{equation*}
\item[(iii)]
Assume GRH for the Dedekind zeta function of $L$, and AHC and PCC for the extension $L/K$.
Then
 \begin{equation*}\label{cheb-MuMuWo}
 \pi_{\cal{C}}(x, L/K)
=
\frac{\#\cal{C}}{[L:K]} \pi(x)
+
\O\left(
(\#\cal{C})^{\frac{1}{2}} \left(\frac{\#\Gal(L/K)^{\#}}{[L:K]} \right)^{\frac{1}{2}}  x^{\frac{1}{2}} \ [K:\Q]^{\frac{1}{2}} 
\log (M(L/K)x)\right).
\end{equation*}
\end{enumerate}
\end{theorem}
\begin{proof}
Part (i) is \cite[Th\'eor\`eme~4, p.~133]{Se81}.
Part (ii) is \cite[Corollary 3.7, p. 265]{MuMuSa88}. 
Part (iii) is \cite[Theorem 1.2, p. 402]{MuMuWo18}.
\end{proof}

In the proofs of our main results, we do not always need the full strength of the effective asymptotic formulae 
in Theorem \ref{chebotarev-GRH-AHC-PCC}. 
For example, we do not use part (ii) of this theorem, which we included  for clarity and for 
comparison with parts (i) and (iii).
 Instead of part (ii), we will use Theorem \ref{chebotarev-upper-bound}, stated below. We will also use
Theorem \ref{functorial-quotient-group-pcc}, stated below, 
which is a consequence of part (iii).

\begin{theorem}\label{chebotarev-upper-bound}
Let $L/K$ be a Galois extension of number fields.
Let $\emptyset \neq \cal{C} \subseteq \Gal(L/K)$ 
be a subset stable under conjugation by elements of $\Gal(L/K)$.
Assume GRH for the Dedekind zeta function of $L$ and AHC for the extension $L/K$.
Then
\begin{equation*}
\pi_{\cal{C}}(x, L/K)
\ll
\frac{\#\cal{C}}{[L:K]} \pi(x)
+
(\#\cal{C})^{\frac{1}{2}} 
\frac{x^{\frac{1}{2}}}{\log x}
\
[K:\Q] \log M(L/K).
\end{equation*}
\end{theorem}
\begin{proof}
This is \cite[Theorem 2.3, p. 5]{Zy15}.
\end{proof}

\begin{theorem}\label{functorial-quotient-group}
Let $L/K$ be a Galois extension of number fields.
Let $\emptyset \neq \cal{C}_1\subseteq \cal{C}_2 \subseteq \Gal(L/K)$ 
be subsets stable under conjugation by elements of $\Gal(L/K)$.
Let $H \leq \Gal(L/K)$ be a subgroup of $\Gal(L/K)$ 
and
let $N \unlhd H$ be a normal subgroup of $H$.
Assume that:
\begin{enumerate}
\item[(i)]
 every element of $\cal{C}_1$ is conjugate to some element of $H$;
 \item[(ii)]
 $N(\cal{C}_2\cap H) \subseteq \cal{C}_2\cap H$;
 \item[(iii)]
 $H/N$ is an abelian group;
 \item[(iv)]
 GRH holds for the Dedekind zeta function of $L^N$.
 \end{enumerate}
Then 
\begin{eqnarray*}
\pi_{\cal{C}_1}(x, L/K) 
\ll
\frac{\#\widehat{\cal{C}_2 \cap H} \cdot \#N}{\#H} \pi(x)
+
\left(\#\widehat{{\cal{C}_2} \cap H}\right)^{\frac{1}{2}} \frac{x^{\frac{1}{2}}}{\log x} \ [L^H :\Q]\log M\left(L^N/L^H\right)
+
[K:\Q] \log M(L/K),
\end{eqnarray*}
where
$\widehat{{\cal{C}_2} \cap H}$ 
denotes the image of $\cal{C}_2 \cap H$ in 
$\frac{H}{N}$ under the canonical projection
$H \twoheadrightarrow \frac{H}{N}$. 
\end{theorem}

\begin{proof}
Using assumption (i) and \cite[Lemma 2.6 (i)]{Zy15},
we obtain that
${\tilde{\pi}}_{ {\cal{C}}_1}(x, L/K) \leq \tilde{\pi}_{ {\cal{C}}_1 \cap H}\left(x, L/L^H\right)$.
Since ${\cal{C}}_1 \subseteq {\cal{C}}_2$, we obviously have 
$\tilde{\pi}_{ {\cal{C}}_1 \cap H}\left(x, L/L^H\right) \leq \tilde{\pi}_{ {\cal{C}}_2 \cap H}\left(x, L/L^H\right)$.
Using assumption (ii), \cite[Lemma 2.6 (ii)]{Zy15}, and  asymptotic formula (\ref{pi-versus-pi-tilde}),
 we obtain that
$$
\tilde{\pi}_{ {\cal{C}}_2 \cap H}\left(x, L/L^H\right) 
= 
\tilde{\pi}_{ \widehat{{\cal{C}}_2 \cap H}}\left(x, L^N/L^H\right)
=
\pi_{ \widehat{{\cal{C}}_2 \cap H}}\left(x, L^N/L^H\right)
+
\O\left([L^H : \Q] \left(\frac{x^{\frac{1}{2}}}{\log x} + \log M\left(L^N/L^H\right) \right)\right).
$$
Using assumption (iii) and \cite{Ar27}, we deduce that AHC holds for the extension $L^N/L^H$.
Then, using assumption (iv) and  Theorem \ref{chebotarev-upper-bound}, we deduce that
$$
\pi_{ \widehat{{\cal{C}}_2 \cap H}}\left(x, L^N/L^H\right)
\ll
\frac{
\# (\widehat{ {\cal{C}}_2 \cap H} )
}{
[H:N]
} 
\pi(x)
+
\# (\widehat{ {\cal{C}}_2 \cap H} )^{\frac{1}{2}}
\
\frac{x^{\frac{1}{2}}}{\log x}
\
[L^H : \Q]
\log M\left(L^N/L^H\right).
$$
Finally, putting all these observations together and using (\ref{pi-versus-pi-tilde}) again to relate
$\pi_{ {\cal{C}}_1}(x, L/K)$
to
$\tilde{\pi}_{ {\cal{C}}_1}(x, L/K)$,
we infer that
\begin{eqnarray*}
\pi_{ {\cal{C}}_1}(x, L/K)
&\ll&
\frac{
\# (\widehat{ {\cal{C}}_2 \cap H} )
}{
[H:N]
} 
\pi(x)
+
\# (\widehat{ {\cal{C}}_2 \cap H} )^{\frac{1}{2}}
\
\frac{x^{\frac{1}{2}}}{\log x}
\
[L^H : \Q]
\log M\left(L^N/L^H\right)
\\
&+&
\frac{x^{\frac{1}{2}}}{\log x} ([L^H:\Q] + [K:\Q])
+
[L^H:\Q] \log M\left(L^N/L^H\right)
+
[K:\Q] \log M(L/K)
\\
&\ll&
\frac{
\# (\widehat{ {\cal{C}}_2 \cap H} )
}{
[H:N]
} 
\pi(x)
+
\# (\widehat{ {\cal{C}}_2 \cap H} )^{\frac{1}{2}}
\
\frac{x^{\frac{1}{2}}}{\log x}
\
[L^H : \Q]
\log M\left(L^N/L^H\right)
+
[K:\Q] \log M(L/K).
\end{eqnarray*}
\end{proof}

\begin{theorem}\label{functorial-quotient-group-pcc}
Let $L/K$ be a Galois extension of number fields.
Let $\emptyset \neq \cal{C} \subseteq \Gal(L/K)$ 
be a subset stable under conjugation by elements of $\Gal(L/K)$.
Let $H \leq \Gal(L/K)$ be a subgroup of $\Gal(L/K)$ 
and
let $N \unlhd H$ be a normal subgroup of $H$.
Assume that:
\begin{enumerate}
\item[(i)]
 every element of $\cal{C}$ is conjugate to some element of $H$;
 \item[(ii)]
 $N(\cal{C}\cap H) \subseteq \cal{C}\cap H$;
 \item[(iii)]
 GRH holds for the Dedekind zeta function of $L^N$;
 \item[(iv)]
AHC and PCC hold for the number field extension $L^N/L^H$.
 \end{enumerate}
Then 
\begin{eqnarray*}
\pi_{ {\cal{C}}}(x, L/K)
&\ll&
\frac{
\# (\widehat{ {\cal{C}} \cap H} )\cdot \# N
}{
\#H
} 
\pi(x)
+
\# (\widehat{ {\cal{C}} \cap H} )^{\frac{1}{2}}
\
\left(\frac{\#\Gal(L^N/L^H)^{\#}}{[H: N]} \right)^{\frac{1}{2}}
\
x^{\frac{1}{2}}
\
[L^H : \Q]^{\frac{1}{2}}
\log \left(M(L^N/L^H) x\right)\\
&+&
[L^H:\Q]\left(\frac{x^{\frac{1}{2}}}{\log x} + \log M(L^N/L^H)\right)+
[K:\Q] \log M(L/K),
\end{eqnarray*}
where
$\widehat{{\cal{C}} \cap H}$ 
denotes the image of $\cal{C} \cap H$ in 
$\frac{H}{N}$ under the canonical projection
$H \twoheadrightarrow \frac{H}{N}$. 
\end{theorem}

\begin{proof}
Similarly to the proof of the previous theorem, 
 we obtain that
\begin{align*}
{\tilde{\pi}}_{ {\cal{C}}}(x, L/K) \leq \tilde{\pi}_{ {\cal{C}} \cap H}(x, L/L^H) 
&= 
\tilde{\pi}_{ \widehat{{\cal{C}} \cap H}}(x, L^N/L^H)\\
&=
\pi_{ \widehat{{\cal{C}} \cap H}}(x, L^N/L^H)
+
\O\left([L^H : \Q] \left(\frac{x^{\frac{1}{2}}}{\log x} + \log M(L^N/L^H) \right)\right).
\end{align*}
Then, using assumptions (iii)-(iv) and part (iii) of Theorem  \ref{chebotarev-GRH-AHC-PCC}, we deduce that
$$
\pi_{ \widehat{{\cal{C}} \cap H}}(x, L^N/L^H)
=
\frac{
\# (\widehat{ {\cal{C}} \cap H} )
}{
[H:N]
} 
\pi(x)
+
\O\left(\# (\widehat{ {\cal{C}} \cap H} )^{\frac{1}{2}}
\
\left(\frac{\#\Gal(L^N/L^H)^{\#}}{[L^N : L^H]} \right)^{\frac{1}{2}}
\
x^{\frac{1}{2}}
\
[L^H : \Q]^{\frac{1}{2}}
\log\left( M(L^N/L^H) x\right)\right).
$$
Putting all these observations together, we infer that
\begin{eqnarray*}
\pi_{ {\cal{C}}}(x, L/K)
&\ll&
\frac{
\# (\widehat{ {\cal{C}}\cap H} )
}{
[H:N]
} 
\pi(x)
+
\# (\widehat{ {\cal{C}} \cap H} )^{\frac{1}{2}}
\
\left(\frac{\#\Gal(L^N/L^H)^{\#}}{[H: N]} \right)^{\frac{1}{2}}
\
x^{\frac{1}{2}}
\
[L^H : \Q]^{\frac{1}{2}}
\log \left(M(L^N/L^H) x\right)
\\
&+&
\frac{x^{\frac{1}{2}}}{\log x} ([L^H:\Q] + [K:\Q])
+
[L^H:\Q] \log M(L^N/L^H)
+
[K:\Q] \log M(L/K)
\\
&\ll&
\frac{
\# (\widehat{ {\cal{C}} \cap H} )
}{
[H:N]
} 
\pi(x)
+
\# (\widehat{ {\cal{C}} \cap H} )^{\frac{1}{2}}
\
\left(\frac{\#\Gal(L^N/L^H)^{\#}}{[H: N]} \right)^{\frac{1}{2}}
\
x^{\frac{1}{2}}
\
[L^H : \Q]^{\frac{1}{2}}
\log \left(M(L^N/L^H) x\right)\\
&+&
[L^H:\Q]\left(\frac{x^{\frac{1}{2}}}{\log x} + \log M(L^N/L^H) \right)+
[K:\Q] \log M(L/K).
\end{eqnarray*}
\end{proof}

In Section 6, we will apply 
Theorems \ref{chebotarev-GRH-AHC-PCC},  \ref{functorial-quotient-group}, and  \ref{functorial-quotient-group-pcc} 
to number field extensions associated to an abelian variety defined over $\Q$.

%\bigskip
%#############################
%#############################
%#############################
\section{Distinguished subgroups in the general symplectic groups over a finite prime field}\label{linear-AG}
%#############################
%#############################
%#############################

In this section, we fix an arbitrary integer $g \geq 1$ and an arbitrary
rational prime $\ell$,
and
turn our attention to particular subgroups of $\GSp_{2g}(\F_{\ell})$.
As we will see in Section 6, these subgroups  give rise to particular
 subextensions of the $\ell$-division field of an abelian variety defined over $\Q$, of dimension $g$.

We will view the elements of $\GSp_{2g}(\F_{\ell})$ using the following block matrix description
\[
\GSp_{2g}(\F_\ell)
=
\left\{
\begin{array}{lll}
                                   &  &-C^{t} A + A^{t} C = 0
\\
M = \begin{pmatrix} A & B \\  C & D \end{pmatrix} \in \GL_{2g}(\F_\ell): & 
A, B, C, D \in M_g(\F_\ell), & 
- C^{t}B+A^tD=\mu I_{g} \ \text{for some} \ \mu\in \F_{\ell}^{\times} 
\\
                                    & & -D^tB+B^tD=0
\end{array}\right\},
\]
where the element $\mu$ above is the multiplicator of the matrix $M$.

Before introducing the subgroups of $\GSp_{2g}(\F_\ell)$ 
needed in the proofs of our main results,
we introduce several 
subgroups of $\GL_{g}(\F_\ell)$, as follows:
 \begin{eqnarray*}
 \cal{B}_g(\F_\ell)
  &:=& 
   \left\{
  \begin{pmatrix} 
  a_{11}  & \cdots & a_{1g}\\
  &  \ddots & \vdots\\
  & & a_{gg}  
  \end{pmatrix}
  \in \GL_{g}(\F_\ell): 
  a_{ii} \in \F_{\ell}^{\times} \ \forall 1 \leq i \leq g,
  a_{ij}\in \F_{\ell} \ \forall 1\leq i<j \leq g \right\};
  \\
   \cal{U}_g(\F_\ell)
  &:=& 
   \left\{
  \begin{pmatrix} 
 1  & \cdots & a_{1g}\\
  &  \ddots & \vdots\\
  & & 1  
  \end{pmatrix}
  \in \GL_{g}(\F_\ell): 
  a_{ij}\in \F_{\ell} \ \forall  1 \leq i < j \leq g  \right\};
  \\
\cal{U}'_g(\F_\ell)
  &:=& 
   \left\{
  \begin{pmatrix} 
 \lambda  & \cdots & a_{1g}\\
  &  \ddots & \vdots\\
  & & \lambda 
  \end{pmatrix}
  \in \GL_{g}(\F_\ell): 
   \lambda \in \F_{\ell}^{\times},
  a_{ij}\in \F_{\ell} \ \forall 1\leq i < j \leq g  \right\};
  \\
   \cal{T}_g(\F_\ell)
  &:=& 
   \left\{
  \begin{pmatrix} 
  a_{11}  & & \\
  &  \ddots & \\
  & & a_{gg}  
  \end{pmatrix}
  \in \GL_{g}(\F_\ell): 
  a_{ii} \in \F_{\ell}^{\times}
  \ \forall 1\leq i \leq g \right\}.
  \end{eqnarray*}
  
  For a matrix
  $ A=\begin{pmatrix} 
 \lambda  & \cdots & a_{1g}\\
  &  \ddots & \vdots\\
  & & \lambda 
  \end{pmatrix}\in  \cal{U}'_g(\F_\ell)$, we set 
  $
  d(A):= \lambda.
  $

Now we introduce four distinguished subgroups of $\GL_{2g}(\F_\ell)$:
 \begin{eqnarray*}
 B_{2g}(\F_\ell)
  &:=& 
   \left\{
  \begin{pmatrix} 
 A & \mu^{-1}AS\\
 0 & \mu (A^t)^{-1}
  \end{pmatrix}
  \in \GL_{2g}(\F_\ell): 
  A \in \cal{B}_g(\F_\ell),
  \mu \in \F_{\ell}^{\times}, S \in M_g(\F_\ell) \ \text{symmetric} \right\};
  \\
   U_{2g}(\F_\ell)
  &:=& 
   \left\{
  \begin{pmatrix} 
 A  & AS\\
 0 & (A^t)^{-1}\\
  \end{pmatrix}
  \in \GL_{2g}(\F_\ell): 
  A \in \cal{U}_g(\F_\ell),
 S \in M_g(\F_\ell) \ \text{symmetric}\right\};
 \\
U'_{2g}(\F_\ell)
 &:=& 
  \left\{
  \begin{pmatrix} 
 A & \mu^{-1}AS\\
 0 & \mu (A^t)^{-1}
 \end{pmatrix}
 \in \GL_{2g}(\F_\ell): 
 A \in \cal{U}'_g(\F_\ell), \mu= d(A)^2, 
 S \in M_g(\F_\ell) \ \text{symmetric}\right\};
  \\
   T_{2g}(\F_\ell)
  &:=& 
   \left\{
  \begin{pmatrix} 
 A  & 0 \\
 0 &  \mu (A^t)^{-1} 
  \end{pmatrix}
  \in \GL_{2g}(\F_\ell): 
  A \in \cal{T}_g(\F_\ell),
  \mu\in \F_{\ell}^{\times} \right\}.
  \end{eqnarray*}

For the rest of the section, we focus on properties of these subgroups
that will be used in the proofs of Theorems \ref{main-thm1} - \ref{main-thm3}.

\begin{proposition}\label{groups-normal}
\
\begin{enumerate}
\item[(i)] $B_{2g}(\F_\ell)$ is a subgroup of $\GSp_{2g}(\F_\ell)$.
\item[(ii)] 
$U_{2g}(\F_\ell)$ is a normal subgroup of $B_{2g}(\F_\ell)$.
\item[(iii)]
 $U'_{2g}(\F_\ell)$  is a normal subgroup of $B_{2g}(\F_\ell)$.
\end{enumerate}
\end{proposition}
\begin{proof}

(i)
Let
\[
M =
\begin{pmatrix}
A' & \mu^{-1}A'S \\
0 & \mu(A'^t)^{-1}
\end{pmatrix}\in B_{2g}(\F_\ell).
\]
Using
the block matrix description of $\GSp_{2g}(\F_{\ell})$ and 
 taking
\[
A=A',\ B=\mu^{-1}A'S, \ C=0, \ D=\mu (A'^t)^{-1},
\] 
we can verify that 
\[
-C^{t} A + A^{t} C = 0, \ - C^{t}B+A^tD=\mu I_{g}, \  -D^tB+B^tD=0.
\]
This shows that
$M\in \GSp_{2g}(\F_{\ell})$. 

Note that for any  matrices
\[
M_1 =
\begin{pmatrix}
A_1 & \mu_1^{-1}A_1S_1 \\
0 & \mu_1(A_1^t)^{-1}
\end{pmatrix}, \ 
M_2 =
\begin{pmatrix}
A_2 & \mu_2^{-1}A_2S_2 \\
0 & \mu_2(A_2^t)^{-1}
\end{pmatrix}
\in B_{2g}(\F_\ell),
\]
we have 
\begin{align*}
M_1M_2^{-1}
&=
\begin{pmatrix}
A_1 & \mu_1^{-1}A_1S_1 \\
0 & \mu_1(A_1^t)^{-1}
\end{pmatrix}
 \begin{pmatrix}
A_2^{-1} & -\mu_2^{-2} S_2 A_2^t \\
0 & \mu_2^{-1} A_2^t
\end{pmatrix} \\
&=
 \begin{pmatrix}
A_1A_2^{-1} & -\mu_2^{-2} A_1S_2 A_2^t+\mu_1^{-1}\mu_2^{-1}A_1SA_2^t \\
0 & \mu_1\mu_2^{-1} ((A_1A_2^{-1})^{t})^{-1}
\end{pmatrix}.
\end{align*}
Then, by setting 
\[
\mu_3 := \mu_1\mu_2^{-1}, \ A_3 := A_1A_2^{-1}, \ S_3 := (-\mu_1\mu_2^{-3}+\mu_2^{-2})A_2S_2A_2^t,
\]
we see that 
\begin{align*}
M_1M_2^{-1}
=
\begin{pmatrix}
A_3 & \mu_3^{-1}A_3S_3 \\
0 & \mu_3(A_3^t)^{-1}
\end{pmatrix} \in B_{2g}(\F_{\ell}).
\end{align*}
This shows that
$B_{2g}(\F_{\ell})$ is 
a group, hence a subgroup of $\GSp_{2g}(\F_{\ell})$.

\medskip

(ii), (iii)
It follows from an  argument 
similar to the one used
in the proof of  (i) that both $U_{2g}(\F_\ell)$ and $U'_{2g}(\F_\ell)$ are subgroups of $B_{2g}(\F_\ell)$.
Note that, for any  matrices
\[
M_1 = 
\begin{pmatrix}
A_1 & \mu^{-1}A_1 S_1 \\
0 & \mu(A_1^t)^{-1}
\end{pmatrix}\in B_{2g}(\F_\ell), \quad 
M_2 = 
\begin{pmatrix} 
 A_2 & A_2 S_2\\
 0 & (A_2^t)^{-1}\\
  \end{pmatrix}
\in U_{2g}(\F_\ell), 
\]
we have
\begin{align*}
M_1^{-1} M_2 M_1
&=
 \begin{pmatrix}
A_1^{-1} & -\mu_1^{-2} S_1 A_1^t \\
0 & \mu_1^{-1} A_1^t
\end{pmatrix} 
\begin{pmatrix} 
 A_2  & A_2 S_2\\
 0 & (A_2^t)^{-1}\\
  \end{pmatrix}
  \begin{pmatrix}
A_1 & \mu_1^{-1} A_1 S_1 \\
0 & \mu_1 (A_1^t)^{-1}
\end{pmatrix} \\
&=
 \begin{pmatrix}
A_1^{-1} A_2 A_1 & \mu_1^{-1}A_1^{-1} A_2 A_1 S_1 +  \mu_1 A_1^{-1} A_2 S_2 (A_1^t)^{-1} - \mu_1^{-1} S_1 A_1^t (A_2^t)^{-1} (A_1^t)^{-1} \\
0 & A_1^t (A_2^{t})^{-1} (A_1^t)^{-1} 
\end{pmatrix} .
\end{align*}
Since $S_1$ is symmetric and since $\cal{U}_g(\ell) \unlhd \cal{B}_g(\F_\ell)$, 
we obtain that $ A_1^{-1} A_2 A_1 \in \cal{U}_g(\F_\ell)$. 
Then, by setting
\[
A_3 :=  A_1^{-1} A_2 A_1, 
\
S_3:=\mu_1^{-1} S_1 + \mu_1 A_1^{-1} S_2 (A_1^t)^{-1} - \mu_1 A_3^{-1} S_1  (A_3^t)^{-1},
%S_3 :=\mu_1^{-1} S_1 + \mu_1 A_1^{-1} S_2 (A_1^t)^{-1} - \mu_1 A_2^{-1} S_1  (A_2^t)^{-1},
\]

we obtain that
\[
M_1^{-1} M_2 M_1
= 
 \begin{pmatrix}
A_3 & A_3 S_3\\
0 & (A_3^t)^{-1}
\end{pmatrix} \in U_{2g}(\F_\ell).
\]
This proves that $U_{2g}(\F_\ell)$ is a normal subgroup of $B_{2g}(\F_\ell)$. 
Since $U_{2g}(\F_\ell) \leq U'_{2g}(\F_\ell)$, 
we deduce  that $U'_{2g}(\F_\ell)$ is a normal subgroup of $B_{2g}(\F_\ell)$.
\end{proof}

\begin{proposition}\label{counting-G(l)-etc}
The following formulae hold:
\begin{eqnarray*}
\#\GSp_{2g}(\F_\ell)
&=& 
(\ell-1)\prod_{1 \leq i \leq g} (\ell^{2i}-1)\ell^{2i-1}; 
 \\
 \#\PGSp_{2g}(\F_\ell)
&=& 
\prod_{1 \leq i \leq g} (\ell^{2i}-1)\ell^{2i-1}; 
 \\
 \#\cal{B}_g(\F_\ell)
 &=& 
 \ell^{\frac{g(g-1)}{2}}(\ell-1)^g;
 \\
  \#\cal{U}_g(\F_\ell)
  &=& 
  \ell^{\frac{g(g-1)}{2}};
 \\
  \#\cal{U}'_g(\F_\ell)
  &=& 
  \ell^{\frac{g(g-1)}{2}}(\ell-1);
  \\
\#\cal{T}_{g}(\F_\ell)
 &=& 
 (\ell-1)^g;
 \\   
 \#B_{2g}(\F_\ell)
 &= &
 \ell^{g^2}(\ell-1)^{g+1};
 \\
 \#U_{2g}(\F_\ell)
 &=&
 \ell^{g^2};
 \\
 \#U'_{2g}(\F_\ell)
 &=&
 \ell^{g^2}(\ell-1);
 \\
 \#T_{2g}(\F_\ell)
 &= & (\ell-1)^{g+1}.
\end{eqnarray*}
\end{proposition}
\begin{proof}
The formulae for $\#\GSp_{2g}(\F_\ell)$
 and $\#\PGSp_{2g}(\F_\ell)$ are well-known 
  (for example, see \cite[Theorem 3.1.2, p. 35]{O'Me78}). 
The formulae for the orders of $\cal{B}_g(\F_{\ell})$, $\cal{U}_g(\F_\ell)$, $\cal{U}'_g(\F_\ell)$ and $\cal{T}_g(\F_\ell)$ are clear from the definitions of the groups. 
The formulae for the orders of $B_{2g}(\F_\ell)$, $U_{2g}(\F_\ell)$, $U'_{2g}(\F_\ell)$ and $T_{2g}(\F_\ell)$ can be obtained from the definitions of the groups by counting the number of each block matrices, as follows:
\begin{eqnarray*}
\#B_{2g}(\F_\ell)
&=&
\#\cal{B}_g(\F_\ell)
\cdot
\#(\Z/\ell\Z)^{\times}
\cdot
\ell^{\frac{g(g+1)}{2}}= \ell^{g^2}(\ell-1)^{g+1};
\\
\#U_{2g}(\F_\ell)
&=&
\#\cal{U}_g(\F_\ell)
\cdot
\ell^{\frac{g(g+1)}{2}}
= 
\ell^{g^2};
\\
\#U'_{2g}(\F_\ell)
&=&
\#\cal{U}_g(\F_\ell)
\cdot
\#(\Z/\ell\Z)^{\times}
\cdot
\ell^{\frac{g(g+1)}{2}}
=  
\ell^{g^2}(\ell-1);
\\
\#T_{2g}(\F_\ell)
&=&
\#\cal{T}_g(\F_\ell)
\cdot
\#(\Z/\ell\Z)^{\times}
= 
(\ell-1)^{g+1}.
\end{eqnarray*}
\end{proof}

\begin{proposition}\label{conj-number}
The following upper bounds hold:
\begin{eqnarray*}
\# \GSp_{2g}(\F_{\ell})^{\#} 
&\ll_g&
 \ell^{g+1}, \\
 \# \PGSp_{2g}(\F_{\ell})^{\#} 
&\ll_g&
 \ell^{g}.
\end{eqnarray*}
\end{proposition}
\begin{proof}
The upper bound for $\#\GSp_{2g}(\F_{\ell})^{\#}$ can be obtained from \cite[(iii), p. 36]{Wa63} and \cite[(2), p.1]{Ga70}. 
To obtain an upper bound for $\#\PGSp_{2g}(\F_{\ell})$, we proceed as follows.
Note that each conjugacy class in $\PGSp_{2g}(\F_{\ell})$ may be viewed as the equivalence class of $\Lambda_{2g}(\F_{\ell})$-orbits of conjugacy classes in $\GSp_{2g}(\F_{\ell})^{\#}$. 
Now fix an arbitrary element 
 $\cal{C}\in \GSp_{2g}(\F_{\ell})^{\#}$. If there is an element $b\in \F_{\ell}^{\times}$ such that $(bI_{2g}) \cal{C}=\cal{C}$, 
 then, by comparing  determinants, we have
$b^{2g}=1$. So $b$ may take at most $2g$ elements.
 By the orbit-stabilizer theorem in group theory, each  $\Lambda_{2g}(\F_{\ell})$-orbit 
 of $\GSp_{2g}(\F_{\ell})^{\#}$ contains at least 
 $\frac{\# \Lambda_{2g}(\F_{\ell})}{2g}$ conjugacy classes. Therefore, 
\[
\# \PGSp_{2g}(\F_{\ell})^{\#} \leq \frac{\#\GSp_{2g}(\F_{\ell})^{\#}}{\frac{\# \F_{\ell}^{\times}}{2g}} \ll_g \ell^g.
\] 
\end{proof}

\begin{proposition}\label{groups-abelian}
The groups $B_{2g}(\F_\ell)/U_{2g}(\F_\ell)$ and $T_{2g}(\F_\ell)$ are isomorphic.
Consequently,  the quotient groups $B_{2g}(\F_\ell)/U_{2g}(\F_\ell)$ and $B_{2g}(\F_\ell)/U'_{2g}(F_\ell)$ are abelian. 
\end{proposition}

\begin{proof}
By comparing the orders of 
$B_{2g}(\F_\ell)/U_{2g}(\F_\ell)$
and
$T_{2g}(\F_\ell)$,
we deduce that the composition  
$T_{2g}(\F_\ell) \hookrightarrow B_{2g}(\F_\ell) \twoheadrightarrow B_{2g}(\F_\ell)/U_{2g}(\F_\ell)$ is an isomorphism. 
Then, 
recalling that
 $T_{2g}(\F_\ell)$  is abelian, 
 we obtain that
 $B_{2g}(\F_\ell)/U_{2g}(\F_\ell)$
 is abelian.
 Furthermore, by
  observing that 
$U_{2g}(\F_\ell) \leq U'_{2g}(\F_\ell)\leq B_{2g}(\F_\ell)$,
we deduce that 
$B_{2g}(\F_\ell)/U'_{2g}(\F_\ell)$ is also abelian. 
\end{proof}

%\bigskip
%#############################
%#############################
%#############################
\section{Distinguished conjugacy classes in the general symplectic groups over a finite prime field}
%#############################
%#############################
%#############################

As in the  previous section, we fix an arbitrary integer $g \geq 1$ and an arbitrary
rational prime $\ell$.
This time, we turn our attention to particular unions of conjugacy classes in $\GSp_{2g}(\F_{\ell})$.

First,
we
 recall that a matrix $M \in M_{2g}(\F_{\ell})$ is called semisimple if its minimal polynomial 
over $\F_{\ell}$
has distinct roots in ${\overline{\F}}_{\ell}$, which is equivalent to the existence of a matrix $N \in \GL_{2g}(\overline{\F}_{\ell})$ such that $N M N^{-1}$ is a diagonal matrix in $M_{2g}(\overline{\F}_{\ell})$.
Next,
we
recall that the eigenvalues of $M$ are the roots of the characteristic polynomial 
$\car_{M}(X) := \det (X I_{2g} - M) \in \F_{\ell}[X]$ of $M$. 
Finally, we recall that,
when $M \in \GSp_{2g}(\F_{\ell})$, its characteristic polynomial $\car_M(X)$ 
has
the form
\begin{equation*}
\car_M(X) = X^{2g} + b_1 X^{2g-1}+ \ldots + b_g X^g + \mu b_{g-1} X^{g-1}+ \ldots + \mu^{g-1} b_1X + \mu^g\in \F_{\ell}[X],
\end{equation*}
where $\mu$ is the multiplicator of $M$. 
Therefore, upon
factoring $\car_M(X)$ over 
$\overline{\F}_{\ell}$,
 we obtain that
there exist $\lambda_{1}(M), \ldots, \lambda_{g}(M) \in \overline{\F}_{\ell}$ such that
\begin{equation*}
\car_{M}(X) = \ds\prod_{1\leq i\leq g}(X-\lambda_i(M)) (X - \mu \lambda_i^{-1}(M)) \in \overline{\F}_{\ell}[X].
\end{equation*}

With notation  as above, we introduce the following subsets associated to  $\GSp_{2g}(\F_{\ell})$,
an integer $t$, and a positive real number $z$:

\begin{eqnarray*}
\cal{C}_0(\ell, t)
&:=& 
\left\{
 M \in \GSp_{2g}(\F_\ell):  
 \ \tr M =  -t (\mod \ell)
\right\};
\\
\cal{C}(\ell, t)
&:=& 
\left\{
 M \in \GSp_{2g}(\F_\ell):  
\lambda_{i}(M) \in \F_{\ell}^{\times} \ 1\leq i\leq g, \ \tr M =  -t (\mod \ell)
\right\};
\\
\cal{C}^{ss}(\ell, t)
&:=& 
\left\{
M \in \cal{C}(\ell, t): M \ \text{semisimple}
\right\};
\\
\cal{C}_{B}(\ell, t) 
&:=& 
{\cal{C}}(\ell, t) \cap B_{2g}(\F_\ell);
\\
\widehat{ \cal{C}_0(\ell, 0)}
&:=&
\; \text{the image of $\cal{C}_0(\ell, 0)$ in  $\PGSp_{2g}(\F_\ell)$};
\\
\widehat{ \cal{C}_{B}(\ell, t)} 
&:=&
\; \text{the image of $\cal{C}_{B}(\ell, t)$ in  $B_{2g}(\F_\ell)/U_{2g}(\F_\ell)$};
\\
\widehat{\cal{C}'_{B}(\ell, t)}
&:=&
\; \text{the image of $\cal{C}_{B}(\ell, t)$ in  $B_{2g}(\F_\ell)/U'_{2g}(\F_\ell)$};
\\
\cal{C}^{ss}(\ell, |t|\leq z)  
&:=&
\; \bigcup_{\substack{t\in \Z\\|t|\leq z}}\cal{C}^{ss}(\ell, t);
\\
\cal{C}(\ell, |t|\leq z)  
&:=&
\; \bigcup_{\substack{t\in \Z\\|t|\leq z}}\cal{C}(\ell, t);
\\
\cal{C}_{B}(\ell, |t|\leq z) 
&:= &
\; \cal{C}(\ell, |t|\leq z) \cap B_{2g}(\F_\ell);
\\
\reallywidehat{\cal{C}_B(\ell, |t|\leq z)} 
&:= &
\;   \text{the image of $\cal{C}_{B}(\ell, |t|\leq z)$ in  $B_{2g}(\F_\ell)/U_{2g}(\F_\ell)$}.
\end{eqnarray*}

For the rest of the section, we focus on  
properties of these subsets 
that
 will be used in the proofs of 
Theorems \ref{main-thm1} - \ref{main-thm3}.

The arguments for Proposition \ref{propr-conj-sets} are straightforward and included for completeness; this proposition ensures that
Theorem \ref{functorial-quotient-group} can be applied in our setting. 
\begin{proposition}\label{propr-conj-sets}
Assume that the  prime $\ell$ satisfies  $\ell\nmid 2g$.
Given any integer $t$ and any positive real number $z$ such that $|t|\leq z$, the following properties hold.
\begin{enumerate}
\item[(i)]
The sets $\cal{C}^{ss}(\ell, t)$ and $\cal{C}^{ss}(\ell, |t|\leq z)$
are
non-empty 
and 
stable under conjugation by elements of $\GSp_{2g}(\F_{\ell})$. 
\item[(ii)]
The sets $\cal{C}_{B}(\ell, t)$ and $\cal{C}_{B}(\ell, |t|\leq z)$
are non-empty 
and
stable under conjugation by elements of $B_{2g}(\F_{\ell})$. 
\item[(iii)]
Every element in $\cal{C}^{ss}(\ell, t) \cup \cal{C}^{ss}(\ell, |t|\leq z)$ 
is conjugate over $\GSp_{2g}(\F_\ell)$ to some element of $B_{2g}(\F_\ell)$. 
\item[(iv)]
We have the set inclusions
$$U_{2g}(\F_\ell) \ {\cal{C}}_{B}(\ell, t) = \cal{C}_{B} (\ell, t);$$ 
$$U'_{2g}(\F_\ell) \ {\cal{C}}_{B}(\ell, 0) = \cal{C}_{B} (\ell, 0);$$
$$U_{2g}(\F_\ell) \ {\cal{C}}_{B}(\ell, |t|\leq z) = \cal{C}_{B} (\ell, |t|\leq z).$$
\end{enumerate}
\end{proposition}

\begin{proof}

(i)
Since $\ell \nmid 2 g$, the inverses $2^{-1} (\mod \ell)$ and $g^{-1} (\mod \ell)$ exist.
We consider the cases $\ell \nmid (t + g)$ and $\ell \mid (t + g)$ separately
and show that, in either case,  ${\cal{C}}^{ss}(\ell, t) \neq \emptyset$.

If $\ell \nmid (t + g)$, then the inverse $(t + g)^{-1} (\mod \ell)$ exists. 
We define
$$M := \diag \left(1 (\mod \ell), \ldots, 1 (\mod \ell), - (t+g)^{-1}  g^{-1}(\mod \ell), - (t+g)^{-1}  g^{-1}(\mod \ell)\right) \in T_{2g}(\F_{\ell}).$$
Since $M$ is semisimple and $\tr M = - t(\mod \ell)$, 
we obtain that $M$
 belongs to ${\cal{C}}^{ss}(\ell, t)$.
 
If $\ell \mid (t + g)$, we define
$$M := \diag \left(2^{-1} (\mod \ell), \ldots, 2^{-1} (\mod \ell), (1 - 2^{-1}) (\mod \ell), (1 - 2^{-1}) (\mod \ell)\right) \in T_{2g}(\F_{\ell}).$$
Since $(1 - 2^{-1}) (\mod \ell)$ is non-zero and 
$g \equiv -t (\mod \ell)$,
we obtain that
$M$ is semisimple and $\tr  M = - t (\mod \ell)$, that is, that $M$ belongs to ${\cal{C}}^{ss}(\ell, t)$.

We will show now that 
the set $\cal{C}^{ss}(\ell, t)$ is 
stable under conjugation by elements of $\GSp_{2g}(\F_{\ell})$. 
First, observe that,
 for any $M\in \cal{C}^{ss}(\ell, t)$ and any $N\in \GSp_{2g}(\F_{\ell})$, we have 
\[\tr (NMN^{-1})=\tr(M)= -t(\mod \ell).\] 
It remains to check 
that
$NMN^{-1}$ is semisimple. 
Since $M$ is semisimple, there exists $A\in \GL_{2g}(\overline{\F}_{\ell})$, such that $AMA^{-1}$ is a diagonal matrix in $M_{2g}(\overline{\F}_{\ell})$. 
Taking $A' := AN^{-1}\in \GL_{2g}(\overline{\F}_{\ell})$, we obtain that
\[
A'(NMN^{-1})A'^{-1}=AMA^{-1}
\] 
is
 a diagonal matrix 
in $M_{2g}(\overline{\F}_{\ell})$. As a consequence, $\cal{C}^{ss}(\ell, t)$ is 
stable under conjugation by elements of $\GSp_{2g}(\F_{\ell})$. 

\medskip

(ii) Since 
 the diagonal matrices in $\cal{C}^{ss}(\ell, t)$ are in ${\cal{C}}_B(\ell, t)$, we 
deduce
from (i) that $\cal{C}_{B}(\ell, t) \neq \emptyset$.
 The set $\cal{C}_B(\ell, t)$ is 
stable under conjugation by elements of $B_{2g}(\F_{\ell})$ because 
both 
$\mathcal{C}(\ell,t)$ and $B_{2g}(\F_{\ell})$ are stable under conjugation by
elements of
 $B_{2g}(\F_{\ell})$. 
The 
similar result about
$\cal{C}_{B}(\ell, |t|\leq z)$ follows by noting that $\cal{C}_{B}(\ell, |t|\leq z)=\bigcup_{\substack{t\in \Z\\|t|\leq z}}\cal{C}_B(\ell, t)$.

\medskip
 
(iii) Let $M \in {\cal{C}}^{ss} (\ell, t)$, with 
multiplicator $\mu$. Then 
$$
\car_M(X) = \ds\prod_{1 \leq i \leq g} (X-\lambda_i(M)) (X - \mu \lambda_i^{-1}(M)) \in \F_{\ell}[X].
$$
Define
$$
M' := \diag\left(\lambda_1(M), \ldots, \lambda_g(M), \mu \lambda_1^{-1}(M), \ldots, \mu \lambda_g^{-1}(M)\right) \in B_{2g}(\F_{\ell}).
$$
Since $M', M$ are semisimple and $\car_{M'} (X) = \car_M (X)$, 
by \cite[Lemma 3.4]{Ch97} we deduce that  $M$ is conjugate over $\GSp_{2g}(\F_\ell)$ to $M'$, which is indeed an element of $B_{2g}(\F_\ell)$.
Since each element $M$ of $\cal{C}^{ss}(\ell, |t|\leq z) $ is an element of $\cal{C}^{ss}(\ell, t)$ for some integer $t$
with $|t| \leq z$,
 by the previous argument we deduce that $M$ is also conjugate over  $\GSp_{2g}(\F_\ell)$ to some element of $B_{2g}(\F_\ell)$. 
 
\medskip

(iv) Let $M_1  \in U_{2g}(\F_{\ell})$ and $M_2  \in {\cal{C}}_{B}(\ell, t)$. 
Then $M_1 M_2 \in B_{2g}(\F_{\ell})$ and $M_1 M_2$ has the same diagonal entries 
as $M_2$. Since $\tr M_2 = - t (\mod \ell)$, we obtain that $\tr (M_1 M_2) = - t (\mod \ell)$.
As such, $U_{2g} (\F_{\ell}) {\cal{C}}_{B}(\ell, t) \subseteq {\cal{C}}_{B}(\ell, t)$. 

Now let $M_1 \in U'_{2g}(\F_{\ell})$ and $M_2 \in {\cal{C}}_{B}(\ell, 0)$.
Denote the  diagonal elements of $M_1$ by $a$.  
Then $M_1 M_2 \in B_{2g}(\F_\ell)$ 
and $\tr (M_1 M_2)= a  \tr M_2 = 0$. 
As such, $U'_{2g}(\F_\ell) \ {\cal{C}}_{B}(\ell, 0) \subseteq \cal{C}_{B} (\ell, 0)$. 

The result about ${\cal{C}}_{B}(\ell, |t|\leq z)$ follows from the observation that
 $$
 U_{2g}(\F_\ell) \ {\cal{C}}_{B}(\ell, |t|\leq z) 
 = 
 \bigcup_{\substack{t\in \Z\\|t|\leq z}} U_{2g}(\F_\ell) \cal{C}_{B} (\ell, t)
 \subseteq 
\bigcup_{\substack{t\in \Z\\|t|\leq z}}  \cal{C}_{B} (\ell, t).
$$

For each case above, the reverse inclusion follows from the fact that the identity is contained in $U_{2g}(\F_\ell)$ and $U'_{2g}(\F_\ell)$. 
\end{proof}

\begin{proposition}\label{counting-C(l)-etc}
Assume that the  prime $\ell$ is odd. 
Given
any integer $t$ and any positive real number $z$ such that $|t|\leq z$, 
 the following upper bounds hold:
$$
\#\cal{C}_{0}(\ell, t) \ll \ell^{2g^2+g};
$$
$$
\#\widehat{\cal{C}_{0}(\ell, 0)} \ll \ell^{2g^2+g-1};
$$
$$
\#\widehat{\cal{C}_{B}(\ell, t)} \ll (\ell-1)^g;
$$
 $$
\#\widehat{\cal{C}'_{B}(\ell, 0)}  \ll  (\ell-1)^{g-1};
 $$
 $$
\#\reallywidehat{\cal{C}_{B}(\ell, |t|\leq z)} \ll (2z+1)(\ell-1)^g.
$$
\end{proposition}
\begin{proof}
We recall from the calculation in \cite[(9), p. 3569]{CoDaSiSt17} that
\[
\frac{\#\cal{C}_{0}(\ell, t) }{\#\GSp_{2g}(\F_{\ell})}=\frac{1}{\ell}+\O\left(\frac{1}{\ell^3}\right).
\]
Then the upper bound for $\#\cal{C}_{0}(\ell, t)$ follows from the asymptotic above and the upper bound
$\#\GSp_{2g}(\F_{\ell})\ll \ell^{2g^2+g+1}$ derived from Proposition \ref{counting-G(l)-etc}.

 Noting that 
 \[
 \Lambda(\F_{\ell}) \cal{C}_{0}(\ell, 0)=\cal{C}_{0}(\ell, 0),
 \]
we obtain that the inverse image of $\widehat{\cal{C}_{0}(\ell, 0)}$ in $\GSp_{2g}(\F_{\ell})$ is $\cal{C}_{0}(\ell, 0)$. Therefore,
 \[
 \# \widehat{\cal{C}_{0}(\ell, 0)}=\frac{\#\cal{C}_{0}(\ell, t)}{\#\F_{\ell}^{\times}} \ll \ell^{2g^2+g-1}.
 \]

Recalling from Proposition \ref{groups-abelian} that $B_{2g}(\F_\ell)/U_{2g}(\F_\ell)\simeq T_{2g}(\F_\ell)$,
we deduce that,  under this bijection, 
$\widehat{\cal{C}_{B}(\ell, t)}$ may be identified with the set of matrices
\[
\{M \in T_{2g}(\F_\ell): \tr M = -t (\mod \ell)\}.
\]
Then, using the definition of $ T_{2g}(\F_\ell)$, we obtain that
\begin{eqnarray*}
 \#\widehat{\cal{C}_{B}(\ell, t)} 
 &\leq & 
 \#\{(a_1, \ldots, a_g, \mu)\in (\F_{\ell}^{\times})^{g+1}: \sum_{1\leq i\leq g}a_i+\mu \sum_{1\leq i\leq g}a_i^{-1} =  -t (\mod \ell)\}
 \\
& = &  \sum_{(a_1, \ldots, a_g)\in (\F_{\ell}^{\times})^g} \#\{\mu\in \F_{\ell}^{\times}: \sum_{1\leq i\leq g}a_i+\mu \sum_{1\leq i\leq g}a_i^{-1} =  -t (\mod \ell)\} 
 \\
 & = &   \sum_{\substack{(a_1, \ldots, a_g)\in (\F_{\ell}^{\times})^g\\ \sum_{1\leq i\leq g}a_i^{-1}\neq 0}} \#\{\mu\in \F_{\ell}^{\times}: \mu  \sum_{1\leq i\leq g}a_i^{-1}=  -t -\sum_{1\leq i\leq g}a_i (\mod \ell)\} + \sum_{\substack{(a_1, \ldots, a_g)\in (\F_{\ell}^{\times})^g\\ \sum_{1\leq i\leq g}a_i^{-1}=0}} (\ell-1)
 \\
  &\leq &
  \sum_{\substack{(a_1, \ldots, a_g)\in (\F_{\ell}^{\times})^g\\ \sum_{1\leq i\leq g}a_i^{-1}\neq 0}} 1+ (\ell-1)^{g-1} \cdot (\ell-1) 
  \\
  &\leq &
2 (\ell-1)^{g}.
\end{eqnarray*}

To estimate  $\#\widehat{\cal{C}'_{B}(\ell, 0)}$,
we
observe first that
 the set $\widehat{\cal{C}_{B}(\ell, 0)}$ surjects onto  $\widehat{\cal{C}'_{B}(\ell, 0)}$ under the reduction map
$B_{2g}(\F_\ell)/U_{2g}(\F_\ell)\twoheadrightarrow B_{2g}(\F_\ell)/U'_{2g}(\F_\ell)$. 
Next, we observe that the inverse image of 
$\widehat{\cal{C}'_{B}(\ell, 0)}$ in $B_{2g}(\F_\ell)$ is of the form $U'_{2g}(\F_\ell)\cal{C}_{B}(\ell, 0)$, which is equal to $\cal{C}_{B}(\ell, 0)$ by applying 
part (iv) of Proposition \ref{propr-conj-sets}.
Thus the inverse image of 
$\widehat{\cal{C}'_{B}(\ell, 0)}$ in $B_{2g}(\F_\ell)/U_{2g}(\F_\ell)$ is $\cal{C}_{B}(\ell, 0)$. 
Consequently,
\[
\#\widehat{\cal{C}'_{B}(\ell, 0)} 
=
\frac{\#\widehat{\cal{C}_{B}(\ell, 0)}}{\#(U'_{2g}(\F_\ell)/U_{2g}(\F_\ell))}
\leq
2(\ell-1)^{g-1}.
\]

Lastly, to estimate the size of $\#\reallywidehat{\cal{C}_{B}(\ell, |t|\leq z)}$, we simply observe that 
\[
\#\reallywidehat{\cal{C}_{B}(\ell, |t|\leq z)}=\sum_{|t|\leq z}\#\widehat{\cal{C}_{B}(\ell, t)}\leq 2(2z+1)(\ell-1)^g.
\]
\end{proof}

%\bigskip
%#############################
%#############################
%#############################
\section{Proofs of Theorems \ref{main-thm1} - \ref{main-thm3}}
%#############################
%#############################
%#############################

In this section, we  prove Theorems \ref{main-thm1} - \ref{main-thm3}
using the preliminary results presented in Sections 3 - 5.

%#############################
%#############################
\subsection{Background on abelian varieties}
%#############################
%#############################

We start with a summary of
  properties of abelian varieties.
We refer the reader to 
the papers \cite{Fa83}, \cite{Ho68}, \cite{Oo08}, \cite{Se85}, \cite{Se86}, \cite{Ta66}, \cite{Wa69}, \cite{Za18}
and 
to the books \cite{CoSi86}, \cite{La83},   \cite{Mu70}
for  the theory of abelian varieties.

We start by considering an abelian variety $A$ defined over an arbitrary field 
$F$
and of dimension $g$.
For an integer $n \geq 1$, we write 
$A[n]$ for the group of $n$-torsion elements of $A(\overline{F})$
and we write
$F(A[n])$ for the extension of $F$ generated by 
$A[n]$.
For a rational prime $\ell \neq  \car F$,
we write $T_{\ell}(A)$ for the $\ell$-adic Tate module of $A$, defined as the inverse limit
$\ds\lim_{\leftarrow \atop{k}} A[\ell^k]$.
We recall that, 
if $\gcd(n, \car F) = 1$, then $A[n]$ is a free $\Z/n \Z$-module of rank $2g$, 
endowed with an action of the absolute Galois group $\Gal(F^{\text{sep}}/F)$,
and
if $\ell \neq \car F$, then $T_{\ell}(A)$ is a free $\Z_{\ell}$-module of rank $2g$, endowed with a continuous action of $\Gal(F^{\text{sep}}/F)$.
These Galois actions  give rise to group homomorphisms
\begin{equation}\label{galois-repres-mod-n}
\Gal(F^{\text{sep}}/F) \longrightarrow \Aut_{\Z/n \Z}(A[n]) 
\end{equation}
and
\begin{equation}\label{galois-repres-ell-adic}
 \Gal(F^{\text{sep}}/F) \longrightarrow \Aut_{\Z_{\ell}}(T_{\ell}(A)), 
\end{equation}
which we call the residual modulo $n$ Galois representation of $A$
and the $\ell$-adic Galois representation of $A$, respectively.
The latter form a compatible system. 

Next, we focus our attention on an abelian variety $A$ defined over $\Q$ and of dimension $g$.
We denote by $N_A$ the conductor of $A$.
For a prime $p \nmid N_A$, we denote by $A_p$ the reduction of $A$ modulo $p$, which is an abelian variety defined over $\F_p$ and of dimension $g$.

Associated to $A$ and an integer $n \geq 1$, we have  the residual modulo $n$ Galois representation
\begin{equation}\label{galois-repres-mod-n-A}
\overline{\rho}_{A, n}: 
\Gal(\overline{\Q}/\Q) \longrightarrow \Aut_{\Z/n \Z}(A[n]),
\end{equation}
about which it is known that
\begin{equation}\label{galois-repres-mod-n-A-inj}
\Q(A[n]) = \overline{\Q}^{\Ker \overline{\rho}_{A, n}}
\end{equation}
and that
\begin{equation}\label{neron-ogg-sha}
\text{
$\overline{\rho}_{A, n}$ is unramified outside the primes dividing $n N_A$.
}
\end{equation}

Associated to $A$ and a prime $\ell$, we have the $\ell$-adic Galois representation
\begin{equation}\label{galois-repres-ell-adic-A}
\rho_{A, \ell}:
 \Gal(\overline{\Q}/\Q) \longrightarrow \Aut_{\Z_{\ell}}(T_{\ell}(A)).
\end{equation}

Associated to $A$, a prime $p \nmid N_A$, and an integer $n \geq 1$ such that $p \nmid n$, we have the residual modulo $n$ Galois representation
\begin{equation}\label{galois-repres-mod-n-A_p}
\overline{\rho}_{A_p, n}: 
\Gal(\overline{\F}_p/\F_p) \longrightarrow \Aut_{\Z/n \Z}(A_p[n]).
\end{equation}

Associated to $A$, a prime $p \nmid N_A$, and another prime $\ell \neq p$, we have the $\ell$-adic Galois representation
\begin{equation}\label{galois-repres-ell-adic-A_p}
\rho_{A_p, \ell}:
 \Gal(\overline{\F}_p/\F_p) \longrightarrow  \Aut_{\Z_{\ell}}(T_{\ell}(A_p)).
\end{equation}

Now let us fix primes $p$ and $\ell$ such that $p \nmid \ell N_A$.
The $p$-th power Frobenius
$\overline{\F}_p \longrightarrow \overline{\F}_p$,
$\alpha \mapsto \alpha^p$,
gives rise 
to an $\F_p$-endomorphism $\pi_p(A) \in \End_{\F_p}(A_p)$ of $A_p$,
which, in turn, gives rise to a $\Z_{\ell}$-linear operator $\left.\pi_p(A)\right|_{T_{\ell}(A_p)}$ on $T_{\ell}(A_p)$ and an $\F_{\ell}$-linear operator $\left.\pi_p(A)\right|_{A_p[\ell]}$ on $A_p[\ell]$,
both of which we shall also call $\pi_p(A)$.
We denote by
$
P_{\pi_p(A)}(X) \in \Z_{\ell}[X]
$
the characteristic polynomial of $\pi_p(A)$ acting on $T_{\ell}(A_p)$
and recall the following important properties.

\begin{proposition}\label{char-poly-pi_p-is-p-Weil}
Let $A$ be an abelian variety defined over $\Q$, of conductor $N_A$, and of dimension $g$.
Let $p$ and $\ell$ be primes such that $p \nmid \ell N_A$.
Then
$P_{\pi_p(A)}(X)$ is a $p$-Weil polynomial of degree $2g$, 
whose coefficients are integers that are independent of the choice of $\ell$.
\end{proposition}
\begin{proof}
See \cite[Theorem 19.1, p. 143]{Mi86}.
\end{proof}

We write 
\begin{eqnarray}\label{char-poly-pi_p-coeff}
P_{\pi_p(A)}(X) &=&
X^{2 g} + a_{1, p}(A) X^{2g -1} + a_{2, p}(A) X^{2g-2} + \ldots + a_{g, p}(A) X^g 
\\
&+&
 p a_{g-1, p}(A) X^{g-1} + p^2 a_{g-2, p}(A) X^{g-2} + \ldots + p^g \in \Z[X].
 \nonumber
\end{eqnarray}
Over $\overline{\Q}$, we write
\begin{equation}\label{char-poly-pi_p-complex}
P_{\pi_p(A)}(X) = \ds\prod_{1 \leq i \leq 2g} (X-\alpha_i) \in {\overline{\Q}}[X],
\end{equation}
where $\alpha_1, \ldots, \alpha_g \in \overline{\Q}$ have the property that,
for any $\sigma \in \Aut(\C)$, 
\begin{equation}\label{RH-p-number}
|\sigma(\alpha_i)| = \sqrt{p} \quad \forall 1 \leq i \leq 2g.
\end{equation}

Note that the discriminant $\disc {P_{\pi_p(A)}(X)}$ of the polynomial $P_{\pi_p(A)}(X)$ is an integer, defined by the product
\begin{equation}\label{char-poly-pi_p-disc-def}
\disc {P_{\pi_p(A)}(X)}
=
\ds\prod_{1 \leq i < j \leq 2g} (\alpha_i - \alpha_j)^2.
\end{equation}
Thanks to (\ref{RH-p-number}),  
the
absolute value 
of this
discriminant satisfies the upper bound
\begin{equation}\label{char-poly-pi_p-disc-upper}
|\disc {P_{\pi_p(A)}(X)}|
\leq
(4 p)^{(2g-1)g}.
\end{equation}

In light of Proposition \ref{char-poly-pi_p-is-p-Weil},
$\pi_p(A)$ may be regarded as a $p$-Weil number, that is, as an algebraic integer for which,
 for any $\sigma \in \Aut(\C)$, we have that
$|\sigma(\pi_p(A))| = \sqrt{p}$.
Associated to this algebraic integer, we have its minimal polynomial
$Q_{\pi_p(A)}(X) \in \Z[X]$ over $\Q$,
the number field $\Q(\pi_p(A))$ that it generates,
 and the Galois closure
$
K_{A, p} 
$
 of $\Q(\pi_p(A))$ in $\overline{\Q}$.
 In what follows, we record properties of the extensions $\Q(\pi_p(A))/\Q$ and $K_{A, p}/\Q$.
 
 \begin{proposition}\label{field-K-pi_p}
 Let $A$ be an abelian variety defined over $\Q$, of conductor $N_A$, and of dimension $g$.
Let $p$ be a prime such that $p \nmid N_A$.
\begin{enumerate}
\item[(i)]
The degrees of $\Q(\pi_p(A))/\Q$ and $K_{A, p}/\Q$ satisfy the upper bounds
$$
[\Q(\pi_p(A)) : \Q] \leq 2g
\  \ \text{and} \ \
[K_{A, p} : \Q] \leq g!  \ 2^g. 
$$
\item[(ii)]
The absolute values of the discriminants of $\Q(\pi_p(A))/\Q$ and $K_{A, p}/\Q$ satisfy the upper bounds
$$
|d_{K_{A, p}}|
\leq
|d_{\Q(\pi_p(A))}|^{g! \  2^{g-1}}
\leq
(4 p)^{g! \  2^{g-1} (2g-1) g}.
$$
\end{enumerate}
 \end{proposition} 
 \begin{proof}
 (i) For the degree of $\Q(\pi_p(A))/\Q$, 
 we
 observe that
 $$[\Q(\pi_p(A)) : \Q] = \deg Q_{\pi_p(A)}(X) \leq \deg P_{\pi_p(A)} (X) = 2g.$$
  For the degree of $K_{A, p}/\Q$,
we sketch the following well-known argument.
 Recall from Proposition  \ref{char-poly-pi_p-is-p-Weil} 
 that $P_{\pi_p(A)}(X)$ is a $p$-Weil polynomial, having $\pi_p(A)$ as one of its roots in $\overline{\Q}$.
 The complex roots
 $\alpha_1, \ldots, \alpha_{2g}$
 of 
 $P_{\pi_p(A)}(X)$,
 introduced in  (\ref{char-poly-pi_p-complex}),
 come in pairs
 $\left(\alpha_1, \frac{\alpha_1}{p}\right)$,
 $\ldots$,
 $\left(\alpha_g, \frac{\alpha_g}{p}\right)$.
 As such,
 the Galois group 
 $\Gal(K_{A, p}/\Q)$
 of the Galois closure of $\Q(\pi_{p}(A))$ in $\overline{\Q}$
embeds into
the subgroup $W_{2g}$ of the permutation group ${\cal{S}}_{2g}$ on $2g$ elements
which induces a permutation on the set of pairs
$(1, 2)$, $\ldots$,   $(2 g - 1, 2g)$.
The group structure of $W_{2g}$ is obtained by noting that
the action of $W_{2g}$ on the above pairs of integers
gives rise to a short exact sequence of groups
$1 \rightarrow H \rightarrow W_{2g} \rightarrow {\cal{S}}_g \rightarrow 1$,
where
$H$ is the group  generated by the transpositions
$(1, 2)$, $\ldots$, $(2g-1, 2g)$
and
where
${\cal{S}}_g$ is the permutation group on $g$ elements.
Thus $H \simeq (\Z/2 \Z)^g$ and 
$W_{2g} \simeq (\Z/2\Z)^g \rtimes {\cal{S}}_g$.
Consequently,
$[K_{A, p} : \Q] \leq \#W_{2g} =  g! \ 2^g$.
For more details, see  \cite[pp. 168--169]{Ch97} or \cite[pp. 2--3]{Do84}.

 \medskip
 
 (ii) On one hand, we have the inclusion  of orders
 $\Z[\pi_p(A)] \subseteq {\cal{O}}_{\Q(\pi_p(A))}$
 in $\Q(\pi_p(A))$.
 By comparing their discriminants, we obtain the relation
 \begin{equation}\label{disc-order}
 \disc Q_{\pi_p(A)}(X) = c_{A, p}^2 d_{\Q(\pi_p(A))},
 \end{equation}
 where 
 \begin{equation}\label{c_A_p}
 c_{A, p} := \# \left({\cal{O}}_{\Q(\pi_p(A))}/\Z[\pi_p(A)]\right) \in \Z.
 \end{equation}
 In particular, we obtain the divisibility
 $$
 d_{\Q(\pi_p(A))} \mid \disc Q_{\pi_p(A)}(X)
 \; \;
 \text{in}
 \; \; \Z.
 $$
 Since we also have the divisibility
 $$
 Q_{\pi_p(A)}(X) \mid P_{\pi_p(A)}(X)
 \; \; 
 \text{in}
 \; \; 
 \Z[X],
 $$
 we deduce the divisibility relations
 \begin{equation}\label{divisibile-disc}
 d_{\Q(\pi_p(A))}
 \mid
 \disc Q_{\pi_p(A)}(X)
 \mid 
 \disc P_{\pi_p(A)}(X)
 \; \;
 \text{in}
 \; \;
 \Z.
 \end{equation}
 
 On the other hand, we have the inclusion of number fields $\Q(\pi_p(A)) \subseteq K_{A, p}$.
 Since $K_{A, p}$ is the smallest normal extension of $\Q(\pi_p(A))$ in $\overline{\Q}$,
 by comparing  discriminants 
 and recalling \cite[p. 327]{La94},
 we obtain the inequality
 \begin{equation}\label{inequality-disc}
 |d_{K_{A, p}}| \leq |d_{\Q(\pi_p(A))}|^{g! \ 2^{g-1}}.
 \end{equation}

 Putting together 
  (\ref{char-poly-pi_p-disc-upper}),
 (\ref{divisibile-disc}),
 and
 (\ref{inequality-disc}),
 we deduce the inequalities claimed in (ii).
 \end{proof}

Fixing a Frobenius element
$\Frob_p \in \Gal(\overline{\Q}/\Q)$
at $p$,
we obtain 
the $\Z_{\ell}$-linear operator
$\rho_{A, \ell}(\Frob_p)$
on $T_{\ell}(A)$ 
and
the $\F_{\ell}$-linear operator
$\overline{\rho}_{A, \ell}(\Frob_p)$
on $A[\ell]$.
We denote by
$P_{A, p}(X) \in \Z_{\ell}[X]$
the characteristic polynomial of $\rho_{A, \ell}(\Frob_p)$ acting on $T_{\ell}(A)$
and by
${\cal{P}}_{A, p}(X) \in \F_{\ell}[X]$
the characteristic polynomial of $\overline{\rho}_{A, \ell}(\Frob_p)$ acting on $A[\ell]$.
We  denote by
${\cal{Q}}_{A, p}(X) \in \F_{\ell}[X]$
the minimal polynomial of $\overline{\rho}_{A, \ell}(\Frob_p)$ acting on $A[\ell]$.
The latter two polynomials satisfy the following two divisibility relations:
\begin{equation}\label{cayley-hamilton}
{\cal{Q}}_{A, p}(X) \mid {\cal{P}}_{A, p}(X)
\; \;
\text{in}
\; \;
\F_{\ell}[X]
\end{equation}
and
\begin{equation}\label{min-char}
{\cal{P}}_{A, p}(X) \mid {\cal{Q}}_{A, p}(X)^{\infty}
\; \;
\text{in}
\; \;
\F_{\ell}[X]
\end{equation}
(see, for example, \cite[Corollary 7.10, Proposition 7.9, p. 376]{Al09}).
In the next two propositions, we relate
 the polynomials $P_{A, p}(X)$ 
 and
 ${\cal{P}}_{A, p}(X)$  to $P_{\pi_p(A)}(X)$,
and
the polynomial ${\cal{Q}}_{A, p}(X)$ to $Q_{\pi_p(A)}(X)$.

\begin{proposition}\label{char-polys}
Let $A$ be an abelian variety defined over $\Q$, of conductor $N_A$, and of dimension $g$.
Let $p$ and $\ell$ be primes such that $p \nmid \ell N_A$.
Then
\begin{enumerate}
\item[(i)]
$P_{A, p}(X) = P_{\pi_p(A)}(X) \in \Z[X]$;
\item[(ii)]
${\cal{P}}_{A, p}(X) = P_{\pi_p(A)}(X) (\mod \ell) \in \F_{\ell}[X]$.
\end{enumerate}
\end{proposition}

\begin{proof}
The $\Z_{\ell}$-module isomorphisms
$$T_{\ell}(A_p) \simeq_{\Z_{\ell}} \Z_{\ell}^{2g} \simeq_{\Z_{\ell}} T_{\ell}(A)$$
and
the $\F_{\ell}$-vector space isomorphisms
$$A_{p}[\ell] \simeq_{\F_{\ell}} \F_{\ell}^{2g} \simeq_{\F_{\ell}} A[\ell]$$
give rise to the following commutative diagram
of ring  homomorphisms,
in which the top horizontal row is  also an isomorphism of $\Z_{\ell}$-algebras
and the bottom horizontal row is also an isomorphism of $\F_{\ell}$-algebras:
\begin{center}
\begin{tikzcd}[column sep=large]
\End_{\Z_{\ell}}(T_{\ell}(A))  \arrow[r, "\simeq_{\Z_{\ell}\text{-alg}}"]\arrow[d, twoheadrightarrow]  &\End(T_{\ell}(A_p)) \arrow[d, twoheadrightarrow] 
\\
\End_{\F_{\ell}}(A[\ell]) \arrow[r, "\simeq_{\F_{\ell}\text{-alg}}"] & \End_{\F_{\ell}}(A_p[\ell]).
\end{tikzcd}
\end{center}
Under this diagram, 
by moving to the right and down, 
the operator
$\rho_{A, \ell}(\Frob_p) \in \End_{\Z_{\ell}}(T_{\ell}(A))$ 
is mapped 
to $\left.\pi_p(A)\right|_{{T_{\ell}}(A_p)} \in \End_{\Z_{\ell}} (T_{\ell}(A_p))$
and then to $\left.\pi_p(A)\right|_{A_p[\ell]} \in \End_{\F_{\ell}}(A_p[\ell])$;
by moving down and to the left,
it is mapped to   $\overline{\rho}_{A, \ell}(\Frob_p)  \in \End_{\F_{\ell}}(A[\ell])$
and then again to $\left.\pi_p(A)\right|_{A_p[\ell]} \in \End_{\F_{\ell}}(A_p[\ell])$.
Hence the polynomial relations (i) and (ii) between the various associated characteristic polynomials hold.
\end{proof}

\begin{proposition}\label{cal-Q-divides-Q}
Let $A$ be an abelian variety defined over $\Q$, of conductor $N_A$, and of dimension $g$.
Let $p$ and $\ell$ be primes such that $p \nmid \ell N_A$.
Then
$$
{\cal{Q}}_{A, p}(X) \mid Q_{\pi_p(A)}(X) (\mod \ell)
\; \;
\text{in}
\; \;
\F_{\ell}[X].
$$
\end{proposition}

\begin{proof}
Viewing $Q_{\pi_p(A)}(X) (\mod \ell) \in \F_{\ell}[X]$ as a polynomial over the ring of $\F_{\ell}$-linear operators on $A_p[\ell]$,
we see that $\left.\pi_p(A)\right|_{A_p[\ell]}$ is one of its roots.
Recall that, under the isomorphism of $\F_{\ell}$-algebras
$\End_{\F_{\ell}}(A_p[\ell]) \longrightarrow \End_{\F_{\ell}}(A[\ell])$,
the operator $\left.\pi_p(A)\right|_{A_p[\ell]}$
is mapped to
$\overline{\rho}_{A, \ell}(\Frob_p)$.
Then, since ${\cal{Q}}_{A, p}(X)$ is the minimal polynomial of $\overline{\rho}_{A, \ell}(\Frob_p)$,
we deduce that
${\cal{Q}}_{A, p}(X) \mid Q_{\pi_p(A)}(X) (\mod \ell)$
as polynomials in $\F_{\ell}[X]$.
\end{proof}

Finally, we
record the following result which will play an important role in the proofs of our main theorems.

\begin{proposition}\label{key-prop}
Let $A$ be an abelian variety defined over $\Q$, of conductor $N_A$, and of dimension $g$.
Let $p$ and $\ell$ be primes such that $p \nmid \ell N_A$.
Assume that $\ell \nmid c_{A, p}$ (defined in (\ref{c_A_p}))
and that $\ell$ splits completely in $\Q(\pi_p(A))$.
Then the $\F_{\ell}$-linear operator
$\overline{\rho}_{A, \ell}(\Frob_p) \in \End_{\F_{\ell}}(A[\ell])$
is semisimple (that is, the minimal polynomial ${\cal{Q}}_{A, p}(X)$ has distinct roots in ${\overline{\F}}_{\ell}$)
and 
has all its eigenvalues in $\F_{\ell}^{\times}$
(that is,  all the roots of the characteristic polynomial ${\cal{P}}_{A, p}(X)$ are in $\F_{\ell}^{\times}$).
\end{proposition}

\begin{proof}
Since $\ell \nmid c_{A, p}$ and $\ell$ splits completely in $\Q(\pi_p(A))$,
by a classical theorem of  Dedekind,
the reduction modulo $\ell$
of the minimal polynomial $Q_{\pi_p(A)}(X)$ of the algebraic integer $\pi_p(A)$
factors into distinct linear factors in $\F_{\ell}[X]$.
Then, using the divisibility relation 
${\cal{Q}}_{A, p}(X) \mid Q_{\pi_p(A)}(X) (\mod \ell)$
as polynomials in $\F_{\ell}[X]$,
proven in Proposition \ref{cal-Q-divides-Q},
we deduce that
${\cal{Q}}_{A, p}(X)$ factors  into distinct linear factors in $\F_{\ell}[X]$.
Since ${\cal{Q}}_{A, p}(X)$ is the minimal polynomial of the $\F_{\ell}$-linear operator
$\overline{\rho}_{A, \ell}(\Frob_p)$ acting on $A[\ell]$,
we obtain that $\overline{\rho}_{A, \ell}(\Frob_p)$ is semisimple. Furthermore, using (\ref{min-char}), we obtain that the characteristic polynomial ${\cal{P}}_{A, p}(X)$ has all its roots in $\F_{\ell}$. 
This means 
that
all 
eigenvalues of 
$\overline{\rho}_{A, \ell}(\Frob_p)$ are in $\F_{\ell}$. Since $\overline{\rho}_{A, \ell}(\Frob_p)$ is invertible, the eigenvalues are in fact in $\F_{\ell}^{\times}$.

\end{proof}

%#############################
%#############################
\subsection{A variation on counting primes with a fixed Frobenius trace}
%#############################
%#############################

The first step in the proofs of Theorem \ref{main-thm1} and part (i) of Theorem \ref{main-thm3}
is the following result established 
 in analogy with \cite[Lemma 18]{CoWa22}.

\begin{proposition}\label{max-lemma-A-generic}
Let $A$ be an abelian variety defined over $\Q$, of conductor $N_A$, and of dimension $g$.
For each $x > 2$, let $y = y(x) > 2$ and  $u = u(x) > 2$ be such that
\begin{equation}\label{u-less-y}
u \leq y.
\end{equation}
Assume that, for any $\varepsilon > 0$,
  \begin{equation}\label{u-greater-log-y}
u \geq 
 y^{\frac{1}{2}} (\log y)^{2 + \varepsilon} 
\end{equation}
and
\begin{equation}\label{log-x-over-log-y}
\ds\lim_{x \rightarrow \infty} \frac{\log x}{(\log y)^{1 + \varepsilon}} = 0.
\end{equation}
Assume RH, as well as GRH for  the number fields 
$K_{A, p}$
defined in Section 6.1,
where $p \nmid N_A$ is an arbitrary prime.
Then, for any $\varepsilon > 0$,
there exists a constant $c(\varepsilon) > 0$ such that:
\begin{enumerate}
\item[(i)]
for any $t \in \Z$ and any sufficiently large $x$,
we have
\begin{eqnarray}\label{max-lemma-bound-first}
&&
\#\left\{p \leq x: p \nmid N_A, a_{1, p}(A) = t\right\}
\\
&\leq&
c(\varepsilon)
\ds\max_{
y \leq \ell \leq y + u
} 
\#\left\{p \leq x: p \nmid \ell N_A, a_{1, p}(A) = t,
\right.
\nonumber
\\
&&
\hspace*{3.5cm}
\left.
\ell \nmid d_{\Q(\pi_p(A))} \# ({\cal{O}}_{\Q(\pi_p(A))}/\Z[\pi_p(A)]),
\ \ell \ \text{splits completely in} \  \Q(\pi_p(A))
\right\}
\nonumber
\end{eqnarray}
\item[(ii)]
for any $z > 0$ and any sufficiently large $x > 0$,
we have
\begin{eqnarray}\label{max-lemma-bound-second}
&&
\ds\sum_{t \in \Z \atop{|t| < z}}
\#\left\{p \leq x: p \nmid N_A, a_{1, p}(A) = t\right\}
\\
&\leq&
c(\varepsilon)
\ds\max_{
y \leq \ell \leq y + u
} 
\ds\sum_{t \in \Z \atop{|t| \leq z}}
\#\left\{p \leq x: p \nmid \ell N_A, a_{1, p}(A) = t,
\right.
\nonumber
\\
&&
\hspace*{4.3cm}
\left.
\ell \nmid d_{\Q(\pi_p(A))} \# ({\cal{O}}_{\Q(\pi_p(A))}/\Z[\pi_p(A)]),
\ \ell \ \text{splits completely in} \  \Q(\pi_p(A))
\right\}.
\nonumber
\end{eqnarray}
\end{enumerate}
\end{proposition}
\begin{proof}
The proof follows the main steps of that of  \cite[Lemma 9]{CoWa22}, as we now explain.
To simplify the exposition, 
for
$x > 0$,
$t \in \Z$,
and
$\ell$  a rational prime,
we use the notation
\begin{equation}\label{notation-pi-A-x-ell-t}
\pi_A(x, \ell, t)
:=
\#\left\{p \leq x: p \nmid \ell N_A, a_{1, p}(A) = t,
\ell \nmid d_{\Q(\pi_p(A))} c_{A, p},
\ \ell \ \text{splits completely in} \  \Q(\pi_p(A))
\right\}
\end{equation}
and
\begin{equation}\label{notation-pi-prime-A-x-ell-t}
\pi'_A(x, \ell, t)
:=
\#\left\{p \leq x: p \nmid \ell N_A, a_{1, p}(A) = t,
\ell \nmid d_{K_{A, p}} c_{A, p},
\ \ell \ \text{splits completely in} \  K_{A, p}
\right\},
\end{equation}
where
$c_{A, p} = \# ({\cal{O}}_{\Q(\pi_p(A))}/\Z[\pi_p(A)])$, as defined in Section 6.1.
Since we have the  inclusion 
of number fields
$\Q(\pi_p(A)) \subseteq K_{A, p}$, we deduce that $d_{\Q(\pi_p(A))}\mid d_{K_{A, p}}$, which implies that
\begin{equation}\label{pi-versus-pi'}
\pi_A'(x, \ell, t)
\leq
\pi_A(x, \ell, t).
\end{equation}
Therefore, it is enough to prove that 
\begin{equation}\label{max-lemma-bound-first-prime}
\#\left\{p \leq x: p \nmid N_A, a_{1, p}(A) = t\right\}
\leq
c(\varepsilon)
\ds\max_{
y \leq \ell \leq y + u
} 
\pi_A'(x, \ell, t)
\end{equation}
and
\begin{equation}\label{max-lemma-bound-second-prime}
\ds\sum_{t \in \Z \atop{|t| < z}}
\#\left\{p \leq x: p \nmid N_A, a_{1, p}(A) = t\right\}
\leq
c(\varepsilon)
\ds\max_{
y \leq \ell \leq y + u
} 
\ds\sum_{t \in \Z \atop{|t| \leq z}}
\pi_A'(x, \ell, t).
\end{equation}

First, observe that
\begin{equation}\label{max-lemma-av-upper}
\ds\sum_{y \leq \ell \leq y + u}
\pi'_A(x, \ell, t)
\leq
(\pi(y+u) - \pi(y))
\ds\max_{y \leq \ell \leq y + u}
\pi'_A(x, \ell, t).
\end{equation}

Next, observe that 
\begin{eqnarray}\label{max-lemma-av-lower-1}
\ds\sum_{y \leq \ell \leq y + u}
\pi'_A(x, \ell, t)
&=&
\ds\sum_{
p \leq x
\atop{
p \nmid N_A
\atop{
a_{1, p}(A) = t
}
}
}
\#\left\{
y \leq \ell \leq y+u:
\ell \nmid p \ d_{K_{A, p}} c_{A, p},
\ \ell \ \text{splits completely in} \ K_{A, p}
\right\}
\\
&=&
\ds\sum_{
p \leq x
\atop{
p \nmid N_A
\atop{
a_{1, p}(A) = t
}
}
}
\#\left\{
y \leq \ell \leq y+u:
\ell \nmid p \ d_{K_{A, p}},
\ \ell \ \text{splits completely in} \ K_{A, p}
\right\}
\nonumber
\\
&-&
\ds\sum_{
p \leq x
\atop{
p \nmid N_A
\atop{
a_{1, p}(A) = t
}
}
}
\#\left\{
y \leq \ell \leq y+u:
\ell \nmid p \ d_{K_{A, p}}, 
\ell \mid  c_{A, p},
\ \ell \ \text{splits completely in} \ K_{A, p}
\right\}
\nonumber
\\
&\geq&
\ds\sum_{
p \leq x
\atop{
p \nmid N_A
\atop{
a_{1, p}(A) = t
}
}
}
\#\left\{
y \leq \ell \leq y+u:
\ell \nmid p \ d_{K_{A, p}},
\ \ell \ \text{splits completely in} \ K_{A, p}
\right\}
\nonumber
\\
&-&
\ds\sum_{
p \leq x
\atop{
p \nmid N_A
\atop{
a_{1, p}(A) = t
}
}
}
\#\left\{
y \leq \ell \leq y+u:
\ell \mid  c_{A, p}
\right\}
\nonumber
\\
&\geq&
\ds\sum_{
p \leq x
\atop{
p \nmid N_A
\atop{
a_{1, p}(A) = t
}
}
}
\#\left\{
y \leq \ell \leq y+u:
\ell \nmid p \ d_{K_{A, p}},
\ \ell \ \text{splits completely in} \ K_{A, p}
\right\}
\nonumber
\\
&-&
\ds\sum_{
p \leq x
\atop{
p \nmid N_A
\atop{
a_{1, p}(A) = t
}
}
}
\nu(c_{A, p}),
\nonumber
\end{eqnarray}
where, for an integer $m \geq 1$, we write $\nu(m)$ for the number of its distinct prime factors.

Using the bound $\nu(m) \leq \frac{\log m}{\log 2}$,
we obtain that
$$
\nu(c_{A, p}) \leq \frac{\log c_{A, p}}{\log 2}.
$$
Using (\ref{char-poly-pi_p-disc-upper}), (\ref{disc-order}) and  (\ref{divisibile-disc}), we obtain  that
$$c_{A, p}^2  \ll (4 p)^{g (2g-1)}.$$
Therefore,
\begin{equation*}\label{disc-index-upper-log}
\nu(c_{A, p}) \ll_{g} \log p. 
\end{equation*}

Using this estimate in the last line of the sequence of inequalities  (\ref{max-lemma-av-lower-1}),
we obtain that
\begin{eqnarray}\label{max-lemma-av-lower-2}
\ds\sum_{y \leq \ell \leq y + u}
\pi'_A(x, \ell, t)
&\geq&
\ds\sum_{
p \leq x
\atop{
p \nmid N_A
\atop{
a_{1, p}(A) = t
}
}
}
\#\left\{
y \leq \ell \leq y+u:
\ell \nmid p \ d_{K_{A, p}},
\ \ell \ \text{splits completely in} \ K_{A, p}
\right\}
\\
&+&
\O_{g}
\left(
\pi_A(x, t) \log x
\right).
\nonumber
\end{eqnarray}

Now we focus on the summands
$\#\left\{
y \leq \ell \leq y+u:
\ell \nmid p \ d_{K_{A, p}},
\ell \ \text{splits completely in} \ K_{A, p}
\right\}$
appearing in (\ref{max-lemma-av-lower-2}).
Recalling that we are assuming that GRH holds for each of the fields $K_{A, p}$,
from the effective version of the Chebotarev Density Theorem
stated in part (i) of Theorem  \ref{cheb-LaOd} we know that
\begin{eqnarray}\label{effective-Cheb-applied}
\#\left\{
y \leq \ell \leq y+u:
\ell \nmid p \ d_{K_{A, p}},
\ell \ \text{splits completely in} \ K_{A, p}
\right\}
&=&
\frac{1}{[K_{A, p} : \Q]}
(\pi(y+u) - \pi(y))
\\
&+&
E_1(y, u, K_{A, p})
+
E_2(y, u, K_{A, p}),
\nonumber
\end{eqnarray}
where
the terms 
$E_1(y, u, K_{A, p})$,
$E_2(y, u, K_{A, p})$
are real valued functions of $y, u,$ and $p$,
that satisfy the upper bounds
\begin{equation}\label{bound-for-E1}
|E_1(y, u, K_{A, p})|
\leq
c_1 (y+u)^{\frac{1}{2}}
\left(
\frac{\log |d_{K_{A, p}}|}{[K_{A, p} : \Q]}
+
\log (y+u)
\right),
\end{equation}
\begin{equation}\label{bound-for-E2}
|E_2(y, u, K_{A, p})|
\leq
c_2 y^{\frac{1}{2}}
\left(
\frac{\log |d_{K_{A, p}}|}{[K_{A, p} : \Q]}
+
\log y
\right),
\end{equation}
with $c_1$, $c_2$ some absolute positive constants.

To obtain an upper bound for the quotient
$\frac{\log |d_{K_{A, p}}|}{[K_{A, p} : \Q]}$,
we invoke
inequality
 (\ref{hensel}):
$$
\frac{\log |d_{K_{A, p}}|}{[K_{A, p} : \Q]}
\leq
\ds\sum_{\ell \mid d_{K_{A, p}}} \log \ell
+
\log [K_{A, p} : \Q]
\leq \log |d_{K_{A, p}}| + \log [K_{A, p} : \Q].
$$
We recall from Proposition \ref{field-K-pi_p} that $[K_{A, p} : \Q] \leq g!  \ 2^g$ and $|d_{K_{A, p}}| \leq (4 p)^{g! \  2^{g-1} (2g-1) g}$. Therefore
$$
\frac{\log |d_{K_{A, p}}|}{[K_{A, p} : \Q]}
\leq
c_3(g)
\log p
$$
for some positive constant $c_3(g)$ that depends on $g$, but not on $p$. 
Using this  bound in (\ref{bound-for-E1}) - (\ref{bound-for-E2}) and then in (\ref{effective-Cheb-applied}),
we obtain that
\begin{eqnarray}\label{effective-Cheb-applied-again}
&&
\#\left\{
y \leq \ell \leq y+u:
\ell \nmid p \ d_{K_{A, p}},
\ \ell \ \text{splits completely in} \ K_{A, p}
\right\}
\\
&=&
\frac{1}{[K_{A, p} : \Q]}
(\pi(y+u) - \pi(y))
+
\O_g\left(
(y+u)^{\frac{1}{2}}
\log (x y)
\right)
\nonumber
\\
&\geq&
\frac{1}{2^g g!}
(\pi(y+u) - \pi(y))
\nonumber
+
\O_g\left(
(y+u)^{\frac{1}{2}}
\log (x y)
\right).
\nonumber
\end{eqnarray}

We now derive a lower bound for $\pi(y+u)-\pi(y)$, where the parameter $u$,  as a function of $y$, the parameter $u$ satisfies (\ref{u-less-y})
and (\ref{u-greater-log-y}).
By the Prime Number Theorem under RH, for any $\varepsilon>0$ (which we choose arbitrarily and keep fixed), we have 
\[
\pi(y+u)-\pi(y)\gg \frac{u}{\log (y+u)}\geq  c_4(\varepsilon) \frac{u}{\log u} > 0
\]
for some positive constant $c_4(\varepsilon)$ depending only on $\varepsilon$ (and not on $u$ or $y$).

Recalling 
(\ref{max-lemma-av-lower-2}), 
we deduce that
\begin{eqnarray}\label{max-lemma-av-lower-3}
\ds\max_{y \leq \ell \leq y + u}
\pi'_A(x, \ell, t)
&\geq&
c_{5}(g, \varepsilon)
\pi_A(x, t)
+
\O_{g, \varepsilon}
\left(
\frac{
(y+u)^{\frac{1}{2}} 
\log (x y) \log u}{u}
\pi_A(x, t) 
\right)
\nonumber
\end{eqnarray}
for some positive constant $c_{5}(g, \varepsilon)$ that depends on $g$ and $\varepsilon$.
Invoking (\ref{u-less-y}), (\ref{u-greater-log-y}), and (\ref{log-x-over-log-y}), we see that, as $x \rightarrow \infty$,
the $\O_{g, \varepsilon}$-term above is $\o(\pi_A(x, t))$.
This completes the proof of (\ref{max-lemma-bound-first}).

To prove (\ref{max-lemma-bound-second}), we use the same argument as for (\ref{max-lemma-bound-first}),
except for replacing 
$\#\left\{p \leq x: p \nmid N_A, a_{1, p}(A) = t\right\}$
with
$\#\left\{p \leq x: p \nmid N_A,| a_{1, p}(A)| \leq z\right\}$.
\end{proof}

\begin{remark}
When invoking this proposition in the proofs of Theorems 1 and 3, we fix $\varepsilon>0$ and take
$y(x)=\frac{x^{\delta}}{(\log x)^{2\delta}}$ and $y(x)= \frac{x^{\delta}}{(\log x)^\varepsilon}$, respectively, for some $\delta\in (0, 1)$.  
 We then choose $u(x)$ to be any function satisfying
 $y(x)^{\frac{1}{2}}(\log y(x))^{2+\varepsilon}\leq u(x) \leq y(x)$.  These choices ensure that assumptions (\ref{u-less-y}), (\ref{u-greater-log-y}), and (\ref{log-x-over-log-y})  are satisfied.

\end{remark}

%#############################
%#############################
\subsection{Proof of Theorem \ref{main-thm1} for arbitrary $t \in \Z$}
%#############################
%#############################

Let $A$ be an abelian variety defined over $\Q$, of conductor $N_A$, and of dimension $g$.
Let $t \in \Z$.
Let $x > 2$ be a real number that goes to infinity.
Our goal in this subsection is to prove the upper bound
for $\pi_A(x, t) = \#\{p \leq x: p \nmid N_A, a_{1, p}(A) = t\}$
claimed in Theorem \ref{main-thm1}, under the main assumptions 
(\ref{thm1-GSp}) and (\ref{RH-GRH}) 
stated below.

It is known that we may always choose a polarization $e$ on $A$, which we now do,
and that, for any sufficiently large prime $\ell$, 
the polarization $e$ gives rise to 
a non-degenerate alternating bilinear form
$e_{\ell}$ on $T_{\ell}(A)$, 
with respect to which
the group 
$\GSp(T_{\ell}(A), e_{\ell})$ of symplectic similitudes 
contains 
 the image $\Im \rho_{A, \ell}$  of $\rho_{A, \ell}$ 
(see \cite[p. 34]{Se85}).
Upon choosing an isomorphism
 $\GSp(T_{\ell}(A), e_{\ell})\simeq \GSp_{2g}(\Z_{\ell})$,
 we may thus assume that
 $\Im \rho_{A, \ell} \leq \GSp_{2g}(\Z_{\ell})$, and hence that
 $\Im \overline{\rho}_{A, \ell} \leq \GSp_{2g}(\F_{\ell})$.

We assume that,  for any sufficiently large prime $\ell$, the residual modulo $\ell$ Galois representation 
$\overline{\rho}_{A, \ell}$ 
has  image  isomorphic to 
$\GSp_{2g}(\F_{\ell})$, 
which implies that
\begin{equation}\label{thm1-GSp}
\Gal(\Q(A[\ell])/\Q) \simeq \GSp_{2g}(\F_{\ell}).
\end{equation}

Considering the subextensions of $\Q(A[\ell])$ fixed by the subgroups $U_{2g}(\F_{\ell})$ and $B_{2g}(\F_{\ell})$ of $\GSp_{2g}(\F_{\ell})$,
$$
\Q \subseteq \Q(A[\ell])^{B_{2g}(\F_{\ell})} \subseteq \Q(A[\ell])^{U_{2g}(\F_{\ell})} \subseteq \Q(A[\ell]),
$$
and recalling that, 
by part (ii) of Proposition \ref{groups-normal},
 $U_{2g}(\F_{\ell})$ is a normal subgroup of $B_{2g}(\F_{\ell})$,
 we obtain the following Galois group structures: 
\begin{equation}\label{thm1-B}
\Gal(\Q(A[\ell])/\Q(A[\ell])^{B_{2g}(\F_{\ell})}) \simeq B_{2g}(\F_{\ell}),
\end{equation}
\begin{equation}\label{thm1-U}
\Gal(\Q(A[\ell])/\Q(A[\ell])^{U_{2g}(\F_{\ell})}) \simeq U_{2g}(\F_{\ell}),
\end{equation}
\begin{equation}\label{thm1-B/U}
\Gal(\Q(A[\ell])^{U_{2g}(\F_{\ell})}/\Q(A[\ell])^{B_{2g}(\F_{\ell})}) \simeq B_{2g}(\F_{\ell})/U_{2g}(\F_{\ell}).
\end{equation}

We assume that
\begin{equation}\label{RH-GRH}
\text{GRH holds}.
\end{equation}
More precisely, we assume the validity of RH and of GRH for the Dedekind zeta functions 
of the number fields $\Q(A[\ell])^{U_{2g}(\F_{\ell})}$ for all sufficiently large primes $\ell$,
as well as
of the number fields $K_{A, p}$ for all primes $p \nmid N_A$.

By (\ref{max-lemma-bound-first}) of Proposition \ref{max-lemma-A-generic},
for any real numbers
$y = y(x) > 2$ and $u = u(x) > 2$
 that depend on and grow with $x$ 
 and that
satisfy
(\ref{u-less-y}),
(\ref{u-greater-log-y}),
and
(\ref{log-x-over-log-y}),
we know that, for any $\varepsilon > 0$, there exists $c(\varepsilon) > 0$ such that
\begin{eqnarray}\label{max-lemma-bound-first-proof}
\pi_A(x, t)
&\leq&
c(\varepsilon)
\ds\max_{
y \leq \ell \leq y + u
} 
\#\left\{p \leq x: p \nmid \ell N_A, a_{1, p}(A) = t,
\right.
\nonumber
\\
&&
\hspace*{3.5cm}
\left.
\ell \nmid d_{\Q(\pi_p(A))} \# ({\cal{O}}_{\Q(\pi_p(A))}/\Z[\pi_p(A)]),
\ \ell \ \text{splits completely in} \  \Q(\pi_p(A))
\right\}.
\end{eqnarray}
We will use this inequality with $x$ sufficiently large
 so that any prime $\ell$ satisfying $y \leq \ell \leq y + u$
is itself sufficiently large 
and (\ref{thm1-GSp}) holds.

By observation (\ref{galois-repres-mod-n-A-inj}) and Proposition \ref{key-prop},
any prime $p$ for which
$\ell \nmid d_{\Q(\pi_p(A))} \# ({\cal{O}}_{\Q(\pi_p(A))}/\Z[\pi_p(A)])$
and
$\ell$ splits completely in  $\Q(\pi_p(A))$
has the property that,
for any prime ideal $\mathfrak{p}$ of $\Q(A[\ell])$ that lies over $p$,
we have that
the matrix $\overline{\rho}_{A, \ell} \left(\left(\frac{\Q(A[\ell])/\Q}{\mathfrak{p}}\right)\right) \in \GL_{2g}(\F_{\ell})$
is semisimple 
and
has all its eigenvalues in $\F_{\ell}$.
By part (ii) of Proposition \ref{char-polys},
any prime $p \nmid \ell N_A$ for which  $a_{1, p}(A) = t$
has the property that,
for any prime ideal $\mathfrak{p}$ of $\Q(A[\ell])$ that lies over $p$,
we have
$\tr \overline{\rho}_{A, \ell} \left(\left(\frac{\Q(A[\ell])/\Q}{\mathfrak{p}}\right)\right) = -t (\mod \ell)$.
Thus, every prime $p$ counted on the right hand side of inequality (\ref{max-lemma-bound-first-proof})
satisfies that,
for any prime ideal $\mathfrak{p}$ of $\Q(A[\ell])$ that lies over $p$,
$$
\overline{\rho}_{A, \ell} \left(\left(\frac{\Q(A[\ell])/\Q}{\mathfrak{p}}\right)\right) \in {\cal{C}}^{ss}(\ell, t).
$$
Combining these observations with (\ref{max-lemma-bound-first-proof}),
we obtain  the inequality
\begin{eqnarray}\label{max-lemma-bound-first-proof-cheb}
\pi_A(x, t)
&\leq&
c(\varepsilon)
\max_{y \leq \ell \leq y + u}
\pi_{{\cal{C}}^{ss}(\ell, t)}(x, \Q(A[\ell])/\Q).
\end{eqnarray}

We estimate 
$\pi_{{\cal{C}}^{ss}(\ell, t)}(x, \Q(A[\ell])/\Q)$
by invoking Theorem \ref{functorial-quotient-group}
with
$K = \Q$,
$L = \Q(A[\ell])$,
$H = B_{2g}(\F_{\ell})$,
$N = U_{2g} (\F_{\ell})$,
${\cal{C}}_1 = {\cal{C}}^{ss}(\ell, t)$,
and
${\cal{C}}_2 = {\cal{C}}(\ell, t)$.
In this case, the hypotheses of Theorem  \ref{functorial-quotient-group}
hold thanks 
to parts (iii) and (iv) of Proposition \ref{propr-conj-sets},
to Proposition \ref{groups-abelian},
and to our GRH assumption. 
We obtain that
\begin{eqnarray*}
\pi_{\cal{C}^{ss}(\ell, t)}(x, \Q(A[\ell])/\Q) 
&\ll&
\frac{\#\widehat{\cal{C}_B(\ell, t)} \cdot \#U_{2g}(\F_{\ell})}{\#B_{2g}(\F_{\ell})} \pi(x)
\\
&+&
\left(\#\widehat{{\cal{C}_B(\ell, t)}}\right)^{\frac{1}{2}} \frac{x^{\frac{1}{2}}}{\log x} 
\ \left[\Q(A[\ell])^{B_{2g}(\F_{\ell})} :\Q\right]\log M\left(\Q(A[\ell])^{U_{2g}(\F_{\ell})} /\Q(A[\ell])^{B_{2g}(\F_{\ell})}\right)
\\
&+&
\log M(\Q(A[\ell])/\Q).
\end{eqnarray*}
Using the Chebyshev bound
$\pi(x) \ll \frac{x}{\log x}$,
 the formulae for
$\#\GSp_{2g}(\F_{\ell})$,
$\#B_{2g}(\F_{\ell})$,
$\# U_{2g}(\F_{\ell})$
recorded in Proposition \ref{counting-G(l)-etc},
and the estimate for
$\#\widehat{{\cal{C}}_B(\ell, t)}$
given in Proposition  \ref{counting-C(l)-etc},
we deduce that, under GRH, 
\begin{eqnarray}\label{upper-bound-C^ss-one}
\pi_{\cal{C}^{ss}(\ell, t)}(x, \Q(A[\ell])/\Q) 
&\ll&
\frac{1}{\ell} \cdot \frac{x}{\log x}
+
\ell^{g^2 + \frac{g}{2}} 
\cdot 
\frac{x^{\frac{1}{2}}}{\log x} 
\cdot 
\log M\left(\Q(A[\ell])^{U_{2g}(\F_{\ell})} /\Q(A[\ell])^{B_{2g}(\F_{\ell})}\right)
\\
&+&
\log M(\Q(A[\ell])/\Q).
\nonumber
\end{eqnarray}

To estimate $\log M(\Q(A[\ell])/\Q)$, we see that, by definition,
\begin{equation}\label{log-M-A-one}
\log M(\Q(A[\ell])/\Q)
=
\log 2
+
\log [\Q(A[\ell]) : \Q]
+ \ds\sum_{p \mid d_{\Q(A[\ell])}} \log p.
\end{equation}

Combining (\ref{neron-ogg-sha})
 and (\ref{log-M-A-one}), and  using the first formula in Proposition \ref{counting-G(l)-etc},
we derive that
\begin{equation}\label{log-M-A-three}
\log M(\Q(A[\ell])/\Q)
\ll_g \log (\ell N_A).
\end{equation}

To estimate $\log M\left(\Q(A[\ell])^{U_{2g}(\F_{\ell})} /\Q(A[\ell])^{B_{2g}(\F_{\ell})}\right)$, we proceed similarly
and derive that
\begin{equation}\label{log-M-A-U-B}
\log M\left(\Q(A[\ell])^{U_{2g}(\F_{\ell})} /\Q(A[\ell])^{B_{2g}(\F_{\ell})}\right)
\ll_g \log (\ell N_A).
\end{equation}

Putting together (\ref{upper-bound-C^ss-one}), (\ref{log-M-A-three}), and (\ref{log-M-A-U-B}),
we obtain the  upper bound
\begin{eqnarray}\label{upper-bound-C^ss-two}
\pi_{\cal{C}^{ss}(\ell, t)}(x, \Q(A[\ell])/\Q) 
&\ll_g&
\frac{1}{\ell} \cdot \frac{x}{\log x}
+
\ell^{g^2 + \frac{g}{2}} 
\cdot 
\frac{x^{\frac{1}{2}}}{\log x} 
\cdot 
\log (\ell N_A),
\end{eqnarray}
which is valid for any sufficiently large prime $\ell$.
Then, recalling (\ref{max-lemma-bound-first-proof-cheb}), we obtain that
\begin{eqnarray}\label{max-lemma-bound-first-proof-cheb-in-ell}
\pi_A(x, t)
&\leq&
c(A, \varepsilon)
\ds\max_{
y \leq \ell \leq y + u
} 
\left(
\frac{1}{\ell} \cdot \frac{x}{\log x}
+
\ell^{g^2 + \frac{g}{2}} 
\cdot 
\frac{x^{\frac{1}{2}}}{\log x} 
\cdot 
\log \ell
\right)
\end{eqnarray}
for some positive constant $c(A, \varepsilon)$ that depends on $A$ and $\varepsilon$.

 Choose
\begin{equation*}\label{choice-y-thm1}
y(x) 
:= 
\frac{
x^{\frac{1}{2 g^2 + g + 2}}
}{
(\log x)^{\frac{2}{2 g^2 + g + 2}} 
}.
\end{equation*}
Fix an arbitrary $\varepsilon>0$ and choose $u(x):=y(x)^{\frac{1}{2}}(\log y(x))^{2+\varepsilon}$. With these choices of $y$ and $u$, the assumptions \eqref{u-less-y}, \eqref{u-greater-log-y}, and \eqref{log-x-over-log-y} of Proposition \eqref{max-lemma-A-generic} are satisfied.
Furthermore, the choice of $y$  minimizes the right hand side of (\ref{max-lemma-bound-first-proof-cheb-in-ell}) by making the two terms in (\ref{max-lemma-bound-first-proof-cheb-in-ell}) asymptotically of the same order.

Noting that  $y\leq \ell \leq y+u\leq 2y$, 
we deduce that
\begin{equation*}
\pi_A(x, t) \ll_{A} \frac{x^{1 - \frac{1}{2g^2+g+2}}}{(\log x)^{1 - \frac{2}{2g^2+g+2}}}.
\end{equation*}
This completes the proof of Theorem \ref{main-thm1} for arbitrary $t$.

%#############################
%#############################
\subsection{Proof of Theorem \ref{main-thm1} for $t = 0$}
%#############################
%#############################

Let $A$ be an abelian variety defined over $\Q$, of conductor $N_A$, and of dimension $g$.
Let $x > 2$ be a real number that goes to infinity.
Our goal in this subsection is to prove the upper bound
for $\pi_A(x, 0) = \#\{p \leq x: p \nmid N_A, a_{1, p}(A) = 0\}$
claimed in Theorem \ref{main-thm1}, under the same two main assumptions 
(\ref{thm1-GSp}) and (\ref{RH-GRH})  
as in Subsection 6.3,
that is,
under the assumption 
that for any sufficiently large prime $\ell$, 
the   residual modulo $\ell$
Galois representation $\overline{\rho}_{A, \ell}$ 
has image isomorphic to  $\GSp_{2g}(\F_{\ell})$ 
and 
under the assumption of
the validity 
of RH and of GRH for the Dedekind zeta functions 
of the number fields $\Q(A[\ell])^{U_{2g}(\F_{\ell})}$ for all sufficiently large primes $\ell$,
as well as
of the number fields $K_{A, p}$ for all primes $p \nmid N_A$.

Considering the subextensions of $\Q(A[\ell])$ fixed by the subgroups $U'_{2g}(\F_{\ell})$ and $B_{2g}(\F_{\ell})$ of $\GSp_{2g}(\F_{\ell})$,
$$
\Q \subseteq \Q(A[\ell])^{B_{2g}(\F_{\ell})} \subseteq \Q(A[\ell])^{U'_{2g}(\F_{\ell})} \subseteq \Q(A[\ell]),
$$
and recalling that, 
by part  (iii) of Proposition \ref{groups-normal},
 $U'_{2g}(\F_{\ell})$ is a normal subgroup of $B_{2g}(\F_{\ell})$,
 in addition to (\ref{thm1-GSp}) and (\ref{thm1-B}),
 we obtain the following Galois group structures
\begin{equation}\label{thm1-U'}
\Gal(\Q(A[\ell])/\Q(A[\ell])^{U'_{2g}(\F_{\ell})}) \simeq U'_{2g}(\F_{\ell}),
\end{equation}
\begin{equation}\label{thm1-B/U'}
\Gal(\Q(A[\ell])^{U'_{2g}(\F_{\ell})}/\Q(A[\ell])^{B_{2g}(\F_{\ell})}) \simeq B_{2g}(\F_{\ell})/U'_{2g}(\F_{\ell}),
\end{equation}
where, by Proposition \ref{groups-abelian}, the quotient $B_{2g}(\F_{\ell})/U'_{2g}(\F_{\ell})$ is abelian. 

Proceeding identically to the proof given in Subsection 6.3, we obtain that
 (\ref{max-lemma-bound-first-proof}) leads to the inequality
\begin{eqnarray}\label{max-lemma-bound-first-proof-cheb-zero}
\pi_A(x, 0)
&\leq&
c(\varepsilon)
\ds\max_{
y \leq \ell \leq y + u
} 
\pi_{{\cal{C}}^{ss}(\ell, 0)}(x, \Q(A[\ell])/\Q),
\end{eqnarray}
where
$\varepsilon > 0$ is arbitrary 
and
where 
$y = y(x) > 2$ and $u = u(x) > 2$
are arbitrary real numbers
 that depend on and grow with $x$ 
 and that
satisfy
(\ref{u-less-y}),
(\ref{u-greater-log-y}),
and
(\ref{log-x-over-log-y}).
Then we estimate 
$\pi_{{\cal{C}}^{ss}(\ell, 0)}(x, \Q(A[\ell])/\Q)$
by invoking Theorem \ref{functorial-quotient-group}
with
$K = \Q$,
$L = \Q(A[\ell])$,
$H = B_{2g}(\F_{\ell})$,
$N = U'_{2g} (\F_{\ell})$,
${\cal{C}}_1 = {\cal{C}}^{ss}(\ell, 0)$,
and
${\cal{C}}_2 = {\cal{C}}(\ell, 0)$.
Similarly to the case considered in Subsection 6.3, the hypotheses of Theorem  \ref{functorial-quotient-group}
hold thanks 
to parts (iii) and (iv) of Proposition \ref{propr-conj-sets},
to Proposition \ref{groups-abelian},
and to our GRH assumption. 
We obtain that
\begin{eqnarray*}
\pi_{\cal{C}^{ss}(\ell, 0)}(x, \Q(A[\ell])/\Q) 
&\ll&
\frac{\#\widehat{\cal{C'}_B(\ell, 0)} \cdot \#U'_{2g}(\F_{\ell})}{\#B_{2g}(\F_{\ell})} \pi(x)
\\
&+&
\left(\#\widehat{{\cal{C'}_B(\ell, 0)}}\right)^{\frac{1}{2}} \frac{x^{\frac{1}{2}}}{\log x} 
\ \left[\Q(A[\ell])^{B_{2g}(\F_{\ell})} :\Q\right]\log M\left(\Q(A[\ell])^{U'_{2g}(\F_{\ell})} /\Q(A[\ell])^{B_{2g}(\F_{\ell})}\right)
\\
&+&
\log M(\Q(A[\ell])/\Q).
\end{eqnarray*}
Using the Chebyshev bound
$\pi(x) \ll \frac{x}{\log x}$,
 the formulae for
$\#\GSp_{2g}(\F_{\ell})$,
$\#B_{2g}(\F_{\ell})$,
$\# U'_{2g}(\F_{\ell})$
recorded in Proposition \ref{counting-G(l)-etc},
and the estimate for
$\#\widehat{{\cal{C'}}_B(\ell, 0)}$
given in Proposition  \ref{counting-C(l)-etc},
we deduce that
\begin{eqnarray}\label{upper-bound-C^ss-one-zero}
\pi_{\cal{C}^{ss}(\ell, 0)}(x, \Q(A[\ell])/\Q) 
&\ll&
\frac{1}{\ell} \cdot \frac{x}{\log x}
+
\ell^{g^2 + \frac{g-1}{2}} 
\cdot 
\frac{x^{\frac{1}{2}}}{\log x} 
\cdot 
\log M\left(\Q(A[\ell])^{U'_{2g}(\F_{\ell})} /\Q(A[\ell])^{B_{2g}(\F_{\ell})}\right)
\\
&+&
\log M(\Q(A[\ell])/\Q).
\nonumber
\end{eqnarray}
Using (\ref{log-M-A-three}) and the estimate 
$\log M\left(\Q(A[\ell])^{U'_{2g}(\F_{\ell})} /\Q(A[\ell])^{B_{2g}(\F_{\ell})}\right)
\ll_g \log (\ell N_A)$,
obtained mutatis mutandis,
we infer that
\begin{eqnarray}\label{upper-bound-C^ss-two-zero}
\pi_{\cal{C}^{ss}(\ell, 0)}(x, \Q(A[\ell])/\Q) 
&\ll_g&
\frac{1}{\ell} \cdot \frac{x}{\log x}
+
\ell^{g^2 + \frac{g-1}{2}} 
\cdot 
\frac{x^{\frac{1}{2}}}{\log x} 
\cdot 
\log (\ell N_A).
\end{eqnarray}
Recalling (\ref{max-lemma-bound-first-proof-cheb-zero}), we obtain that 
\begin{eqnarray}\label{max-lemma-bound-first-proof-cheb-in-el-zerol}
\pi_A(x, 0)
&\leq&
c(A, \varepsilon)
\ds\max_{
y \leq \ell \leq y + u
} 
\left(
\frac{1}{\ell} \cdot \frac{x}{\log x}
+
\ell^{g^2 + \frac{g-1}{2}} 
\cdot 
\frac{x^{\frac{1}{2}}}{\log x} 
\cdot 
\log \ell
\right)
\end{eqnarray}
for some positive constant $c(A, \varepsilon)$ that depends on $A$ and $\varepsilon$.

Choose
\begin{equation*}\label{choice-y-thm1-zero}
y(x) 
:= 
\frac{
x^{\frac{1}{2 g^2 + g + 1}}
}{
(\log x)^{\frac{2}{2 g^2 + g + 1}}
}.
\end{equation*}
Fix an arbitrary $\varepsilon>0$, and choose $u(x):=y(x)^{\frac{1}{2}}(\log y(x))^{2+\varepsilon}$. Note that 
assumptions  (\ref{u-less-y}), (\ref{u-greater-log-y}), and (\ref{log-x-over-log-y}) of Proposition \ref{max-lemma-A-generic} are satisfied. 
Furthermore, noting that  $y\leq \ell \leq y+u\leq 2y$,  
we deduce that
\begin{equation*}
\pi_A(x, t) \ll_{A} \frac{x^{1 - \frac{1}{2g^2+g+1}}}{(\log x)^{1 - \frac{2}{2g^2+g+1}}}.
\end{equation*}
This completes the proof of Theorem \ref{main-thm1} for $t = 0$.

%#############################
%#############################
\subsection{Proof of Theorem \ref{main-thm2} for arbitrary $t \in \Z$}
%#############################
%#############################

Let $A$ be an abelian variety defined over $\Q$, of conductor $N_A$, and of dimension $g$.
Let $x > 2$ be a real number that goes to infinity.
Our goal in this subsection is to prove the upper bound
for $\pi_A(x, t) = \#\{p \leq x: p \nmid N_A, a_{1, p}(A) = t\}$
claimed in Theorem \ref{main-thm2}, under assumptions 
(\ref{thm1-GSp}) and (\ref{RH-GRH}), as in Subsection 6.3,
together  with the assumptions that AHC and PCC hold for the extension $\Q(A[\ell])/\Q$,
 where $\ell$ is a sufficiently large arbitrary prime.

As in Section 6.3, we recall that,
by part (ii) of Proposition \ref{char-polys},
any prime $p \nmid \ell N_A$ for which  $a_{1, p}(A) = t$
has the property that,
for any prime ideal $\mathfrak{p}$ of $\Q(A[\ell])$ that lies over $p$,
$\tr \overline{\rho}_{A, \ell} \left(\left(\frac{\Q(A[\ell])/\Q}{\mathfrak{p}}\right)\right) = -t (\mod \ell)$.
Thus, every prime $p$ counted in $\pi_A(x, t)$
satisfies that,
for any prime ideal $\mathfrak{p}$ of $\Q(A[\ell])$ that lies over $p$,
$$
\overline{\rho}_{A, \ell} \left(\left(\frac{\Q(A[\ell])/\Q}{\mathfrak{p}}\right)\right) \in {\cal{C}}_0(\ell, t).
$$
As such, we have the inequality
\begin{equation}\label{direct-mod-ell}
\pi_A(x, t)
\leq
\pi_{{\cal{C}}_0(\ell, t)}(x, \Q(A[\ell])/\Q).
\end{equation}

By invoking Theorem \ref{functorial-quotient-group-pcc}
with 
$K = \Q$,
$L = \Q(A[\ell])$,
$H = \GSp_{2g} (\F_{\ell})$,
$N = \{I_{2g}\}$,
and
${\cal{C}} = {\cal{C}}_0(\ell, t)$, we obtain that, under GRH, AHC, and PCC,
\begin{eqnarray*}
\pi_{\cal{C}_0(\ell, t)}(x, \Q(A[\ell])/\Q) 
&\ll&
\frac{\#\cal{C}_0(\ell, t)}{\#\GSp_{2g}(\F_{\ell})} \pi(x)
\\
&+&
\#\cal{C}_0(\ell, t)^{\frac{1}{2}} 
\
\left(
\frac{
\#\GSp_{2g}(\F_{\ell})^{\#}
}{
\#\GSp_{2g}(\F_{\ell})
} 
\right)^{ \frac{1}{2} }
\
x^{ \frac{1}{2} } 
\ 
\log \left( M(\Q(A[\ell])/\Q) x \right)
\\
&+&
\frac{ 
x^{\frac{1}{2} }
}{
\log x
}
+
\log M(\Q(A[\ell])/\Q).
\end{eqnarray*}
Then, by applying the upper bounds 
for $\#\GSp_{2g}(\F_{\ell})$ from Proposition \ref{counting-G(l)-etc}, 
for $\#{\cal{C}}_0(\ell, t)$ from Proposition \ref{counting-C(l)-etc}, 
for $\#\GSp_{2g}(\F_{\ell})^{\#}$ from Proposition \ref{conj-number},
 as well as 
   estimate (\ref{log-M-A-three}), 
 we deduce that
\begin{eqnarray}\label{upper-bound-C^ss-one-pcc}
\pi_{\cal{C}_0(\ell, t)}(x, \Q(A[\ell])/\Q) 
&\ll_g&
\frac{1}{\ell} \cdot \frac{x}{\log x}
+
\ell^{\frac{g}{2}}
\cdot 
x^{\frac{1}{2}}
\cdot 
\log (\ell N_A x).
\end{eqnarray}
Finally,
by (\ref{direct-mod-ell}) and choosing
\begin{equation*}
\ell(x)
\asymp
\frac{
x^{\frac{1}{ g + 2}}
}{
(\log x)^{\frac{4}{g + 2}} 
},
\end{equation*}
we 
 derive that
\begin{equation*}
\pi_A(x, t) \ll_{A} \frac{x^{1 - \frac{1}{g+2}}}{(\log x)^{1 - \frac{4}{g+2}}}.
\end{equation*}
This completes the proof of Theorem \ref{main-thm2} for arbitrary $t$.

%#############################
%#############################
\subsection{Proof of Theorem \ref{main-thm2} for $t = 0$}
%#############################
%#############################

Let $A$ be an abelian variety defined over $\Q$, of conductor $N_A$, and of dimension $g$.
Let $x > 2$ be a real number that goes to infinity.
Our goal in this subsection is to prove the upper bound
for $\pi_A(x, 0) = \#\{p \leq x: p \nmid N_A, a_{1, p}(A) = 0\}$
claimed in Theorem \ref{main-thm2}, 
under  assumptions 
(\ref{thm1-GSp}) and (\ref{RH-GRH}), 
together with the assumptions that AHC and PCC hold for the extension $\Q(A[\ell])^{\Lambda(\F_{\ell})}/\Q$,
 where $\ell$ is a sufficiently large arbitrary prime.

As in the proof of Theorem \ref{main-thm2} for arbitrary $t$, our starting point is  inequality (\ref{direct-mod-ell})
for $t = 0$, that is, 
\begin{equation}\label{direct-mod-ell-zero}
\pi_A(x, 0)
\leq
\pi_{{\cal{C}}_0(\ell, 0)}(x, \Q(A[\ell])/\Q).
\end{equation}
We estimate the right hand side by invoking Theorem  \ref{functorial-quotient-group-pcc}
with 
$K = \Q$, 
$L = \Q(A[\ell])$,
$H = \GSp_{2g}(\F_{\ell})$,
$N = \Lambda(\F_{\ell})$,
and
${\cal{C}} = {\cal{C}}_0(\ell, 0)$.
We obtain that, under GRH, AHC, and PCC,
\begin{eqnarray*}
\pi_{\cal{C}_0(\ell, 0)}(x, \Q(A[\ell])/\Q) 
&\ll&
\frac{\# \widehat{\cal{C}_0(\ell, 0)}}{\#\PGSp_{2g}(\F_{\ell})} \pi(x)
\\
&+&
\# \widehat{\cal{C}_0(\ell, 0)}^{\frac{1}{2}} 
\
\left(
\frac{
\#\PGSp_{2g}(\F_{\ell})^{\#}
}{
\#\PGSp_{2g}(\F_{\ell})
} 
\right)^{ \frac{1}{2} }
\
x^{ \frac{1}{2} } 
\ 
\log \left( M(\Q(A[\ell])^{\Lambda(\F_{\ell})}/\Q) x \right)
\\
&+&
\frac{ 
x^{\frac{1}{2} }
}{
\log x
}
+
\log M(\Q(A[\ell])^{\Lambda(\F_{\ell})}/\Q) + \log M(\Q(A[\ell])/\Q).
\end{eqnarray*}
Then, by applying 
the upper bounds for $\#\PGSp_{2g}(\F_{\ell})$ from Proposition \ref{counting-G(l)-etc}, 
for $\#\widehat{{\cal{C}}_0(\ell, 0)}$ from Proposition \ref{counting-C(l)-etc}, 
and for $\#\PGSp_{2g}(\F_{\ell})^{\#}$ from Proposition \ref{conj-number}, 
and 
by observing that
\[
\log M(\Q(A[\ell])^{\Lambda(\F_{\ell})}/\Q)\leq \log M(\Q(A[\ell])/\Q),
\]
 we deduce that
\begin{eqnarray}\label{upper-bound-C^ss-one-pcc}
\pi_{\cal{C}_0(\ell, 0)}(x, \Q(A[\ell])/\Q) 
&\ll_g&
\frac{1}{\ell} \cdot \frac{x}{\log x}
+
\ell^{\frac{g-1}{2}}
\cdot 
x^{\frac{1}{2}}
\cdot 
\log (\ell N_A x).
\end{eqnarray}
Finally, by (\ref{direct-mod-ell-zero}) and choosing 
\begin{equation*}\label{choice-y-thm1-pcc}
\ell(x) 
\asymp
\frac{
x^{\frac{1}{ g + 1}}
}{
(\log x)^{\frac{4}{g + 1}} 
},
\end{equation*}
we  derive that
\begin{equation*}
\pi_A(x, 0) \ll_{A} \frac{x^{1 - \frac{1}{g+1}}}{(\log x)^{1 - \frac{4}{g+1}}}.
\end{equation*}
This completes the proof of Theorem \ref{main-thm2} for $t = 0$.

%#############################
%#############################
\subsection{Proof of part (i) of Theorem \ref{main-thm3}}
%#############################
%#############################

Once again, let $A$ be an abelian variety defined over $\Q$, of conductor $N_A$, and of dimension $g$.
We  keep the two main assumptions (\ref{thm1-GSp}) and (\ref{RH-GRH})  
as in 
Subsections 6.3  and 6.4. 
Our goal in this subsection is to prove that,
for any $\varepsilon > 0$,
the lower bound
$|a_{1, p}(A)| \geq p^{\frac{1}{2g^2+g+1}-\varepsilon}$
holds for a set of primes $p$ of density 1.
The proof proceeds similarly to that of \cite[Theorem 3]{CoWa22}, as follows.

Let $x > 2$ be a real number that goes to infinity.
Let $z = z(x) > 0$, $y = y(x) > 2$, and $u = u(x) > 2$ 
be real numbers that depend on and grow with $x$ 
and  satisfy
(\ref{u-less-y}),
(\ref{u-greater-log-y}),
and
(\ref{log-x-over-log-y}).
By (\ref{max-lemma-bound-second}) of Proposition \ref{max-lemma-A-generic},
we know that, for an arbitrary fixed $\varepsilon > 0$, there exists $c(\varepsilon) > 0$ such that
\begin{eqnarray*}\label{max-lemma-bound-second-proof}
\ds\sum_{t \in \Z \atop{|t| \leq z}}
\pi_A(x, t)
&\leq&
c(\varepsilon)
\ds\max_{
y \leq \ell \leq y + u
} 
\ds\sum_{t \in \Z \atop{|t| \leq z}}
\#\left\{p \leq x: p \nmid \ell N_A, a_{1, p}(A) = t,
\right.
\nonumber
\\
&&
\hspace*{4.3cm}
\left.
\ell \nmid d_{\Q(\pi_p(A))} \# ({\cal{O}}_{\Q(\pi_p(A))}/\Z[\pi_p(A)]),
\ \ell \ \text{splits completely in} \  \Q(\pi_p(A))
\right\}.
\end{eqnarray*}
Proceeding similarly to the proofs of (\ref{max-lemma-bound-first-proof-cheb}) and (\ref{upper-bound-C^ss-two}), 
we obtain that
\begin{eqnarray*}
\ds\sum_{t \in \Z \atop{|t| \leq z}}
\pi_A(x, t)
&\leq&
c(\varepsilon)
\ds\max_{
y \leq \ell \leq y + u
} 
\pi_{ {\cal{C}}^{ss}(\ell, |t| \leq z)}(x, \Q(A[\ell])/\Q)
\nonumber
\\
&\leq&
c(A, \varepsilon)
\ds\max_{
y \leq \ell \leq y + u
} 
\left(
\frac{1}{\ell} \cdot \frac{x z}{\log x}
+
\ell^{g^2 + \frac{g}{2}}
\cdot
\frac{
x^{\frac{1}{2}}
z^{\frac{1}{2}}
}{\log x}
\cdot
\log \ell
\right)
\end{eqnarray*}
for some positive constant $c(A, \varepsilon)$ that depends on $A$ and $\varepsilon$.

Now we choose the parameters. Let  $u=y^{\frac{1}{2}}(\log y)^{2+\varepsilon}$ and note that 
$$\max_{
y \leq \ell \leq y + u
} 
\left(
\frac{1}{\ell} \cdot \frac{x z}{\log x}
+
\ell^{g^2 + \frac{g}{2}}
\cdot
\frac{
x^{\frac{1}{2}}
z^{\frac{1}{2}}
}{\log x}
\cdot
\log \ell
\right)\ll 
\frac{1}{y} \cdot \frac{x z}{\log x}
+
y^{g^2 + \frac{g}{2}}
\cdot
\frac{
x^{\frac{1}{2}}
z^{\frac{1}{2}}
}{\log x}
\cdot
\log y.
$$

Next, we set
\[
y=y(z, x)=\frac{(xz)^{\frac{1}{2g^2+g+2}}}{(\log (xz))^{\frac{2}{2g^2+g+2}}}
\]
 so that the two terms in the last displayed equation are asymptotic  of the same order.
 Observe that 
assumptions  (\ref{u-less-y}), (\ref{u-greater-log-y}), and (\ref{log-x-over-log-y}) of  Proposition \ref{max-lemma-A-generic} are now satisfied.  
Finally, we choose
$$z(x) := 
\frac{
x^{\frac{1}{2 g^2 + g + 1}}
}{
(\log x)^{\frac{2}{2g^2+g+1}+ \varepsilon(1+\frac{1}{2g^2+g+1})}}
$$
so that
\begin{equation*}\label{bound-for-sum}
  \ds\sum_{
 t \in \Z
 \atop{
 |t| 
 \leq 
 z(x)
 }
 }
 \pi_A(x, t)
=
\o(\pi(x)).
 \end{equation*} 

Using this estimate, we deduce that for any $\varepsilon>0$, 
\begin{eqnarray*}
\pi(x)
&=&
\#\left\{
p \leq x:
p \nmid N_A,
|a_{1, p}(A)| > \frac{ p^{\frac{1}{2 g^2 + g + 1}}  }{ (\log p)^{\frac{2}{2g^2+g+1}+ \varepsilon}   }
\right\}
\\
&+&
\#\left\{
p \leq x:
p \nmid N_A,
|a_{1, p}(A)| \leq \frac{ p^{\frac{1}{2 g^2 + g + 1}}  }{ (\log p)^{\frac{2}{2g^2+g+1}+ \varepsilon}   }
\right\}
+
\#\{p \leq x: p \mid N_A\}
\\
&=&
\#\left\{
p \leq x:
p \nmid N_A,
|a_{1, p}(A)| > \frac{ p^{\frac{1}{2 g^2 + g + 1}}  }{ (\log p)^{\frac{2}{2g^2+g+1}+ \varepsilon}   }
\right\}
+
\o(\pi(x)).
\end{eqnarray*} 
This completes the proof of part (i) of Theorem \ref{main-thm3}.

%#############################
%#############################
\subsection{Proof of part (ii) of Theorem \ref{main-thm3}}
%#############################
%#############################
Yet again, let $A$ be an abelian variety defined over $\Q$, of conductor $N_A$, and of dimension $g$.
We  keep the four main assumptions (\ref{thm1-GSp}), (\ref{RH-GRH}), AHC, and PCC,
 as in
Subsections 6.5  and 6.6. 
Our goal in this subsection is to prove that,
for any $\varepsilon > 0$,
the lower bound
$|a_{1, p}(A)| \geq p^{\frac{1}{g+2}-\varepsilon}$
holds for a set of primes $p$ of density 1.
The proof proceeds similarly to that of part (i), as follows.

Let $x > 2$ be a real number that goes to infinity
and let $\varepsilon > 0$.
Using Theorem \ref{main-thm2} , we  deduce that 
\begin{eqnarray*}
\pi(x)
&=&
\#\left\{
p \leq x:
p \nmid N_A,
|a_{1, p}(A)| >p^{\frac{1}{g + 2}-\varepsilon} 
\right\}
\\
&+&
\#\left\{
p \leq x:
p \nmid N_A,
|a_{1, p}(A)| \leq p^{\frac{1}{ g + 2}-\varepsilon} 
\right\}
+
\#\{p \leq x: p \mid N_A\}
\\
&\leq &
\#\left\{
p \leq x:
p \nmid N_A,
|a_{1, p}(A)| >  p^{\frac{1}{ g + 2}-\varepsilon}  
\right\}\\
&+&
\#\sum_{\substack{t\in \Z\\|t|\leq x^{\frac{1}{ g + 2}-\varepsilon}}}\left\{
p \leq x:
p \nmid N_A,
|a_{1, p}(A)| =t
\right\}
+
\o(\pi(x))
\\
&=&
\#\left\{
p \leq x:
p \nmid N_A,
|a_{1, p}(A)| >p^{\frac{1}{g + 2}-\varepsilon} 
\right\}+
 \O_{A, \varepsilon}\left(x^{\frac{1}{ g + 2}-\varepsilon} \cdot \frac{x^{1-\frac{1}{g+2}}}{(\log x)^{1-\frac{4}{g+2}}} \right)\\
&=&
\#\left\{
p \leq x:
p \nmid N_A,
|a_{1, p}(A)| >p^{\frac{1}{g + 2}-\varepsilon} 
\right\}+o(\pi(x)).
\end{eqnarray*} 
This completes the proof of  part (ii) of Theorem \ref{main-thm3}.

%\bigskip
%#############################
%#############################
%#############################
\section{Final remarks}
%#############################
%#############################
%#############################

We conclude with brief remarks about the current conditional approaches towards
obtaining  upper bounds for 
$\pi_A(x, t)$ for a given generic abelian variety $A$ defined over $\Q$ and
of dimension $g$, 
where by generic we mean that $\Im \rho_A$ is open in $\GSp_{2g}(\hat{\Z})$.

The main goal of the present paper is
to prove the currently best upper bound for $\pi_A(x, t)$
under GRH.
Theorem \ref{main-thm1} does provide such bounds, substantially improving upon 
\cite[Theorem 1, pp. 3560-3561]{CoDaSiSt17} for any value of $g$ and any value of $t$.
The secondary goal of the paper  is
to prove the currently best upper bound for $\pi_A(x, t)$
under other sensible unproven hypotheses. 
Theorem \ref{main-thm2} does provide such bounds, substantially improving upon
\cite[Corollaire 17, p. 51]{Be16}.
However, while in the first comparison, both results assume some version of GRH,
in the second comparison, the two results assume different hypotheses.
Indeed, Theorem \ref{main-thm2} assumes GRH, AHC, and PCC,
in contrast with \cite[Corollaire 17, p. 51]{Be16}, which assumes GRH and AHC.
It is plausible that the methods of these two results may be combined to produce an improvement to 
the upper bounds for $\pi_A(x, t)$
 under GRH, AHC, and PCC.
 We relegate such work to a future project.

In the case $g = 1$,
a different conditional approach for proving upper bounds for  $\pi_A(x, t)$ was initiated in \cite{Mu85}.
Therein, 
Murty assumed 
an effective version of the Sato-Tate Conjecture for $A$
and used it to deduce the upper bound $\pi_A(x, t) \ll_{A, t} x^{1 - \frac{1}{4}} (\log x)^{\frac{1}{2}}$.
He also proved that the  effective version of the Sato-Tate Conjecture for $A$ invoked in the above bound
follows from the assumptions that 
 each of the symmetric power L-functions of $A$
 has  an analytic continuation to $\C$,
 satisfies   an appropriate functional equation,
 and satisfies an analogue of  the Riemann Hypothesis.
Thanks to recent work of Newton and Thorne \cite{NeTh21},
the assumption about the analytic continuation  is now known to hold.
Murty's approach was exploited further by Rouse and Thorne \cite{RoTh16}, 
who proved, 
under similar hypotheses as those of Murty's,
together
with the additional assumption  that the conductor $N_A$ of the elliptic curve $A$ is squarefree,
 that
$\pi_A(x, t) \ll_{A, t} \frac{x^{1 - \frac{1}{4}} }{(\log x)^{\frac{1}{2}}}$,
a bound which 
improves on that of \cite{Mu85}, but  not on that of \cite{MuMuWo18}.
In the case $g \geq 2$,
Bucur, Fit\'{e} and Kedlaya  \cite{BuFiKe20} proved 
an analogue of Murty's result concerning the validity of an 
 effective version of the Sato-Tate Conjecture for $A$,
 under assumptions similar to the ones of \cite{Mu85}, together with the additional  assumption that
  the Mumford-Tate Conjecture holds  for the abelian variety $A$.
  Following the strategy of \cite{Mu85},  the result of \cite{BuFiKe20} seems to give
 $\pi_A(x, t) \ll_{A, t} x^{1 - \frac{1}{2 g^2 + 2g}} (\log x)^c$ for some (possibly  non-negative) constant $c$, 
  a bound which does not improve upon Theorem \ref{main-thm1}.
  We plan to explore this approach in depth in future work.
  
  Finally, let us note that similar methods may be used to obtain upper bounds for $\pi_A(x, t)$ for other types
  of abelian varieties $A$. For example,
  in the case of an abelian variety isogenous to the product  $E_1 \times \ldots \times E_g$
  of elliptic curves $E_1, \ldots, E_g$ defined over $\Q$, pairwise non-isogenous over $\overline{\Q}$,
  and each with a trivial $\overline{\Q}$-endomorphism ring,
   the refined mod $\ell$ method of \cite{MuMuSa88}, under GRH, was used in \cite{CoWa22},
   while the direct mod $\ell$ method of \cite{MuMuWo18}, under GRH, AHC, and PCC,
   as well as the Sato-Tate method of \cite{Mu85}, under assumptions similar to the case $g =1$,
   will be used in an upcoming paper of the authors.

\bigskip

{\small{

}

\end{document}